\documentclass[11pt]{article}
\usepackage{a4wide}

\usepackage{color}

\usepackage{amssymb, amsmath}%, amsthm}

\newcommand\id{\operatorname{id}}
\newcommand\diag{\operatorname{diag}}
\newcommand\N{{\mathbb N}}
\newcommand\R{{\mathbb R}}
\newcommand\K{\mathcal{K}}
\renewcommand\L{\mathcal{L}}
\newcommand\KL{\mathcal{KL}}
\newcommand\Ki{\mathcal K_\infty}

\providecommand{\dabs}[1]{|#1|}%{\left\| #1 \right\|}		% Euklidian Norm (double-absolut)
\providecommand{\dabsu}[1]{\left\| #1 \right\|_\infty}	% Input-Norm
\providecommand{\dabsinf}[1]{\left\| #1 \right\|_\infty}	% Input-Norm
\providecommand{\abs}[1]{| #1 |}				% Infinity Norm (absolut)

\newtheorem{theorem}{Theorem}[section]
\newtheorem{corollary}[theorem]{Corollary}
\newtheorem{lemma}[theorem]{Lemma}
\newtheorem{proposition}[theorem]{Proposition}
\newtheorem{definition}[theorem]{Definition}
\newtheorem{remark}[theorem]{Remark}
\newtheorem{assumption}[theorem]{Assumption}
\newtheorem{example}[theorem]{Example}
\newtheorem{proc}[theorem]{Procedure}

\newenvironment{proof}
{{\par\parskip1.0ex %plus0.5ex minus0.5ex 
\textbf{Proof.}} \setlength{\leftskip}{0.0em}}% Setlength rueckt die Beweise um 1.1 em ein
{$\hfill \square$\parskip1.0ex %plus0.5ex minus0.5ex
\par \vspace{0.5em}}

\title{Relaxed ISS Small-Gain Theorems for Discrete-Time Systems}

\author{Roman Geiselhart\thanks{R. Geiselhart is with the Faculty of Computer Science and Mathematics, University of Passau, Innstrasse 33, 94032 Passau, Germany; Email: roman.geiselhart@uni-passau.de.}
\and
Fabian R. Wirth\thanks{F. Wirth is with 
the Faculty of Computer Science and Mathematics, University of Passau, Innstrasse 33, 94032 Passau, Germany; Email: fabian.(lastname)@uni-passau.de.}
}

\begin{document}
\maketitle

\begin{abstract}
In this paper ISS small-gain theorems for discrete-time systems are stated, which do not require input-to-state stability (ISS) of each subsystem. This approach weakens conservatism in ISS small-gain theory, 
 and for the class of exponentially ISS systems we are able to prove that the proposed relaxed small-gain theorems are non-conservative in a sense to be made precise. 
 The proofs of the small-gain theorems rely on the construction of a dissipative finite-step ISS Lyapunov function which is introduced in this work. 
Furthermore, dissipative finite-step ISS Lyapunov functions, as relaxations of ISS Lyapunov functions, are shown to be sufficient and necessary to conclude ISS of the overall system. 
\end{abstract}

\begingroup
\small \textbf{Keywords:} 
input-to-state stability, Lyapunov methods,  small-gain conditions, discrete-time nonlinear systems, large-scale interconnections
\endgroup

%%%%%%%%%%%%%%%%%%%%%%%%%%%%%%%%%%%%%%%%%%%%%%%%%%%%%%%%%%%%%%%%
\section{Introduction}\label{sec:intro}

Large-scale systems form an important class of systems with various applications such as formation control, logistics, consensus dynamics,
networked control systems to name a few.
While stability conditions for such large-scale systems have already been studied in the 1970s and early 1980s cf. \cite{moylan1978stability,Sil79,Vid81} based on linear gains and Lyapunov techniques, nonlinear approaches are more recent. 
An efficient tool in the analysis of large-scale nonlinear control systems is the concept of input-to-state stability (ISS) as introduced in \cite{Son89}, and the introduction of ISS Lyapunov functions in \cite{Son89,SW95}. 
The concept of ISS was originally formulated for continuous-time systems, but has also been established for discrete-time control systems (\cite{KT94,JW01-Automatica}) of the form $x(k+1) = G(x(k),u(k))$, as considered in this work. 

As ISS Lyapunov functions are assumed to decrease at each step (while neglecting the input) the search for ISS Lyapunov functions is a hard task, in general.
To relax this assumption we introduce the concept of dissipative finite-step ISS Lyapunov functions, where the function is assumed to decrease after a finite time, rather than at each time step.
This approach originates from \cite{AP98} and was recently used in \cite{LDA13,GLW13-SG,GGLW13-SCL} for the stability analysis of discrete-time systems without inputs.

We provide, in a first step, an equivalent characterization of input-to-state stability in terms of the existence of a dissipative finite-step ISS Lyapunov function. 
The sufficiency part follows the lines of \cite[Lemma~3.5]{JW01-Automatica}, which shows that the existence of a continuous (dissipative) ISS Lyapunov function implies ISS of the system. 
Necessity is shown using a converse ISS Lyapunov theorem \cite{JW01-Automatica,LA05-ISS-discrete}. Moreover, for the case of exponential ISS systems, we show that any norm is a dissipative finite-step ISS Lyapunov function.

In this paper, we follow the nomenclature of \cite{Sil79} and say that 
a large-scale system is defined through the interconnection of a number of
smaller components. For such systems there exist small-gain type
conditions guaranteeing the ISS property for the interconnected system. 
Whereas small-gain theorems have a long history, the first ISS versions in a trajectory-based formulation and in a Lyapunov-based formulation are given in \cite{JPT94} and \cite{JMW96}, respectively. In both cases the results are stated for systems consisting of two subsystems. These results have been extended to large-scale interconnections in  \cite{DRW07} and \cite{DRW09}. 

The first ISS small-gain theorems for discrete-time systems were presented in~\cite{JW01-Automatica}, which parallel the results of \cite{JPT94} and \cite{JMW96} for continuous-time systems.
For interconnections consisting of more than two subsystems, small-gain theorems are presented in \cite{JLW08-CCC} and in \cite{DRW07}, whereas in \cite{JLW08-CCC}  ISS was defined in a maximum formulation and in \cite{DRW07} the results are given in a summation formulation. 
Further extensions to the formulation via maximization or summation are ISS formulations via \emph{monotone aggregation functions}. 
In this formulation, the ISS small-gain results are shown to hold in a
more general form, see~\cite{Diss-Rueffer}. 
In \cite{DRW09} the authors present an ISS small-gain theorem in a
Lyapunov-based formulation that allows to construct an overall ISS
Lyapunov function. This paper also discusses various examples showing that
depending on the system class different formulations of the ISS property
are natural.
In \cite{LHJ12} a discrete-time version in a maximum formulation is shown
and ISS Lyapunov functions for the overall system are constructed.

The classical idea of ISS small-gain theory is that the interconnection of ISS subsystems is ISS if the influence of the subsystems on each other is small enough. 
This is a sufficient criterion, but the requirement of all systems being ISS is not necessary, even for linear systems as we recall in Section~\ref{sec:sgt}. 
Hence, classical small-gain theorems come with a certain conservatism. 
The main purpose of this work to reduce conservatism in ISS small-gain theory. 
This is achieved in the sense that we will not require each subsystem to be ISS. Indeed, each subsystem may be unstable when decoupled from the other subsystems. 
This is a crucial difference to classical ISS small-gain results, where it is implicitly assumed that the other subsystems act as perturbations. Here the subsystems may have a stabilizing effect on each other.

The requirement imposed is that Lyapunov-type functions for the subsystems have to decrease after a finite time. 
This relaxation also includes previous ISS small-gain theorems and applies to a larger class of interconnected systems. 
Furthermore, if the overall system is expISS, i.e., solutions of the unperturbed system are decaying exponentially, the ISS small-gain theorems are indeed non-conservative, i.e., they are also necessary.

The proof of the ISS small-gain theorems presented give further insight in the systems behavior. 
For the sufficiency part a dissipative finite-step ISS Lyapunov function is constructed from the Lyapunov-type functions and the gain functions involved. 
On the other hand, for expISS systems suitable Lyapunov-type and gain functions are derived. 
This in particular implies a constructive methodology for applications.
However, if the overall system is ISS but not expISS the application of the results is challenging to implement.

To illustrate this methodology we consider a nonlinear discrete-time system consisting of two subsystems that are not both ISS. 
To the best of the authors' knowledge previous ISS small-gain theorems do not apply in this situation.
Following the proposed methodology we derive a dissipative finite-step ISS Lyapunov function to show ISS of the overall system.

The outline of this work is as follows. 
The preliminaries are given in Section~\ref{sec:prel}, followed by the problem statement including the definition of a  dissipative finite-step ISS Lyapunov function,  in Section~\ref{sec:prob}. 
The sufficiency of the existence of dissipative finite-step ISS Lyapunov functions to conclude ISS is stated in Section~\ref{sec:ftISS}. 
In this section we also state a particular converse finite-step ISS Lyapunov theorem that shows that for any expISS system any norm is a dissipative finite-step ISS Lyapunov function.
Section~\ref{sec:sgt} contains the main results. 
First, in Section~\ref{subsec:dissSGT} ISS small-gain theorems are presented that do not require each system to admit an ISS Lyapunov function. 
In Section~\ref{subsec:converseSGT}, we show the non-conservativeness of the relaxed ISS small-gain theorems for the class of expISS systems, by stating a converse of the presented ISS small-gain theorems.
We close the paper in Section~\ref{sec:example} with an illustrative example.

%%%%%%%%%%%%%%%%%%%%%%%%%%%%%%%%%%%%%%%%%%%%%%%%%%%%%%
\section{Preliminaries}
\label{sec:prel}
%%%%%%%%%%%%%%%%%%%%%%%%%%%%%%%%%%%%%%%%%%%%%%%%%%%%%%

%%%%%%%%%%%%%%%%%%%%%%%%%%%%%%%%%%%%%%%%%%%%%%%%%%%%%%
\subsection{Notation and conventions}
\label{subsec:Notation}
%%%%%%%%%%%%%%%%%%%%%%%%%%%%%%%%%%%%%%%%%%%%%%%%%%%%%%

By $\N$ we denote the natural numbers and we assume $0 \in \N$.
Let $\R$ denote the field of real numbers, $\R_+$ the set of nonnegative
real numbers and $\R^n$  the vector space of real column vectors of length $n$; further  $\R^n_+:=(\R_+)^n$ denotes the positive orthant.
For any vector $v \in \R^n$ we denote by $[v]_i$ its $i$th component. Then $\R^n$ induces a partial order for vectors $v,w \in \R^n$ as follows:
We define $v \geq w : \iff [v]_i \geq [w]_i$ and $v>w:\iff [v]_i > [w]_i$, each for all $i\in\{1, \ldots, n\}$. Further; $v \not \geq w :\iff $ there exists an index $i \in \{1, \ldots, n\}$ such that $[v]_i < [w]_i$.

By $\dabs{\cdot}$ we denote an arbitrary fixed monotonic norm on $\R^n$, i.e., if $v,w \in \R^n_+$ with $w\geq v$ then $\dabs{w} \geq \dabs{v}$.
For $x_i \in \R^{n_i}$, $i \in \{1, \ldots, N\}$ let $(x_1,\ldots,x_N):= (x_1^\top, \ldots, x_N^\top)^\top$.
For a sequence $\{u(k)\}_{k \in \N}$ with $u(k) \in \R^m$,  
we define $\dabsu{u}:= \sup_{k \in \N} \{ \dabs{u(k)} \} \in \R_+ \cup \{\infty\}$. If $u(\cdot)$ is bounded, i.e., $\dabsu{u}<\infty$, then $u(\cdot) \in l^\infty (\R^m)$.

We will make use of the following consequence of the equivalence of norms in~$\R^n$:
For any norm $\dabs{\cdot}$ on $\R^n$ there exists a constant $\kappa \geq 1$ such that for all $x=(x_1, \ldots, x_N) \in \R^n$ with $x_i \in \R^{n_i}$ and $n= \sum_{i=1}^N n_i$, it holds
\begin{equation}\label{eq:norm} 
\dabs{x} \leq \kappa \max_{i \in \{1, \ldots, N\}} \dabs{x_i},
\end{equation}
where $\dabs{x_i} := \dabs{(0,\ldots, 0, x_i,0 \ldots, 0)}$.
In particular, if $\dabs{\cdot}$ is a $p$-norm then $\kappa = N^{1/p}$ is the smallest constant satisfying~\eqref{eq:norm}.

%%%%%%%%%%%%%%%%%%%%%%%%%%%%%%%%%%%%%%%%%%%%%%%%%%%%%%
\subsection{Comparison functions and induced monotone maps}
\label{subsec:Comparison}
%%%%%%%%%%%%%%%%%%%%%%%%%%%%%%%%%%%%%%%%%%%%%%%%%%%%%%

It has become standard to use \emph{comparison functions} to state stability properties of nonlinear systems. 
Here we use functions of class $\K,\Ki,\L,\KL$. For a definition see \cite{Kellet-comparison}.

A $\Ki$-function $\alpha$ is called \emph{sub-additive}, if for any $s_1,s_2 \in \R_+$ it holds 
\begin{equation*}\alpha(s_1+s_2) \leq \alpha(s_1)+\alpha(s_2). \end{equation*}

In the following lemma we collect some facts about $\Ki$-functions, which are useful not only in the particular proofs of this work. Note that the symbol $\circ$ denotes the composition of two functions.

\begin{lemma}
\begin{enumerate}
\item \cite[Prop.~3]{GW14KiEquations} The pair $(\Ki,\circ)$ is a non-commutative group. In particular, for $\alpha,\alpha_1, \alpha_2 \in \Ki$ the inverse $\alpha^{-1} \in \Ki$ exists, and $\alpha_1 \circ \alpha_2 \in \Ki$.
\item For $\alpha_1, \alpha_2, \alpha_3 \in \Ki$ we have
\begin{equation*}
 \alpha_1 (\max \{ \alpha_2, \alpha_3 \}) =  \max \{ \alpha_1 \circ \alpha_2, \alpha_1 \circ \alpha_3 \}.
\end{equation*}
\end{enumerate}
\end{lemma}

For any two functions $\alpha_1, \alpha_2: \R_+ \rightarrow \R_+$ we write $\alpha_1 < \alpha_2$ (resp. $\alpha_1 \leq \alpha_2$) if $\alpha_1(s) < \alpha_2(s)$ (resp. $\alpha_1(s) \leq \alpha_2(s)$) for all $s>0$. A continuous function $\eta: \R_+ \rightarrow \R_+$ is called \emph{positive definite}, if $\eta(0)=0$ and $\eta(s)>0$ for all $s>0$. By $\id$ we denote the identity function $\id(s)=s$ for all $s \in \R_+$, and by $0:\R_+ \mapsto 0$ we denote the zero function.

Given $\gamma_{ij} \in \Ki \cup \{0\}$ for $i,j\in\{1, \ldots, n\}$, we define the map $\Gamma_\oplus : \R^n_+ \rightarrow \R^n_+$ by
\begin{equation}\label{eq:gainmap}
 \Gamma_\oplus (s) := \left ( \begin{array}{c} \max \left\{ \gamma_{11}([s]_1), \ldots, \gamma_{1n}([s]_n) \right\} \\ \vdots \\ \max \left\{ \gamma_{n1}([s]_1), \ldots, \gamma_{nn}([s]_n) \right\} \end{array}\right).\end{equation}
For the $k$th iteration of this map we write $\Gamma_\oplus^k$. Note that $\Gamma_\oplus$ is \emph{monotone}, i.e.,  $\Gamma_\oplus(s^1) \leq \Gamma_\oplus(s^2)$ for all $s^1,s^2 \in \R^n_{+}$ with $s^1 \leq s^2$, also $\Gamma_\oplus(0)=0$.

%%%%%%%%%%%%%%%%%%%%%%%%%%%%%%%%%%%%%%%%%%%%%%%%%%%%
\subsection{Small-gain conditions} 
\label{subsec:SGC}
%%%%%%%%%%%%%%%%%%%%%%%%%%%%%%%%%%%%%%%%%%%%%%%%%%%%

Consider the map $\Gamma_\oplus$ from \eqref{eq:gainmap}, let $\delta_i \in \Ki, D_i=(\id+\delta_i)$, $i\in \{1,\ldots, n\}$, and define the diagonal operator $D: \R^{n}_+ \rightarrow \R^{n}_+$ by 
\begin{equation}\label{eq:DiagOp}
 D(s) := \left( D_1([s]_1), \ldots, D_{n}([s]_{n}) \right)^\top.
\end{equation}

\begin{definition}\label{Def:SGC} The map $\Gamma_\oplus$ 
from \eqref{eq:gainmap} satisfies the \emph{small-gain condition} if 
 \begin{equation}\label{eq:sgc}
  \Gamma_\oplus(s) \not \geq s \quad \text{ for all }s \in \R^n_+\backslash\{0\}.
 \end{equation}
The map $\Gamma_\oplus$ satisfies the \emph{strong small-gain condition} if there exists a diagonal operator $D$ as in \eqref{eq:DiagOp} such that 
\begin{equation}\label{eq:ssgc}
 (D\circ \Gamma_\oplus)(s) \not \geq s \quad \text{ for all }s \in \R^n_+\backslash\{0\}.
\end{equation}
\end{definition}

The condition $\Gamma_\oplus(s) \not \geq s$ for all $s\in \R^n_+\backslash \{0\}$, or for short $\Gamma_\oplus \not \geq \id$, means that for any $s>0$ the the image $\Gamma_\oplus(s)$ is decreasing in at least one component $i^* \in \{1,\ldots, n\}$, i.e., $[\Gamma_\oplus(s)]_{i^*}<[s]_{i^*}$. 
Furthermore, we can assume that all functions $\delta_i\in\Ki$ of the diagonal operator $D$ are identical, by setting $\delta(s):=\min_i \delta_i(s)$. 
We will then write $D=\diag(\id + \delta)$. 
For any factorization $D=D_{II} \circ D_{I}$ with diagonal operators $D_I,D_{II}:\R^n_+ \to \R^n_+$ as defined above, it holds that $D \circ \Gamma_\oplus \not \geq \id \iff D_{I}\circ \Gamma_\oplus \circ D_{II} \not \geq \id$, and, in particular, $D \circ \Gamma_\oplus \not \geq \id \iff \Gamma_\oplus \circ D \not \geq \id$.

As shown in \cite[Theorem~5.2]{DRW09} the (strong) small-gain condition~\eqref{eq:sgc} (resp.~\eqref{eq:ssgc}) implies the existence of a so-called $\Omega$-path $\tilde \sigma$ with respect to $\Gamma_\oplus$ (resp. $D\circ \Gamma_\oplus$) (\cite[Definition~5.1]{DRW09}). 
The essential property of $\Omega$-paths used in this work is that $\tilde \sigma = (\tilde \sigma_1, \ldots, \tilde \sigma_n) \in \Ki^n$, and satisfies $\Gamma_\oplus(\tilde \sigma(r))< \tilde \sigma(r)$  (resp. $(D\circ \Gamma_\oplus)(\tilde \sigma(r))< \tilde \sigma(r)$) for all $r>0$.
The numerical construction of $\Omega$-paths can be performed using the algorithm proposed in~\cite{GW12-MCSS}. 
A simple calculation shows that if $\tilde \sigma$ is an $\Omega$-path with respect to $D \circ \Gamma_\oplus$ if and only if $D_{II}^{-1}\circ \tilde \sigma$ is an $\Omega$-path with respect to $D_{I}\circ \Gamma_\oplus \circ D_{II}$, where $D=D_{II}\circ D_{I}$ is split as above.

\begin{remark}
Condition~\eqref{eq:sgc} originates from \cite{DRW07} and is in fact
equivalent to the equilibrium $s^* = 0$ of the system $s(k+1)=
\Gamma_\oplus(s(k))$ being globally asymptotically stable
(GAS\footnote{see Remark~\ref{rem:0GAS} for a definition}), see
\cite[Theorem~6.4]{Ruf09Monotone}.
The idea comes from the linear case with $\Gamma \in \R^{n \times n}_+$, where the following is equivalent (see \cite[Lemma~1.1]{Ruf09Monotone} and \cite[Section~4.5]{DRW07}):
\begin{enumerate}
\item for the spectral radius it holds $\rho(\Gamma)<1$;
\item $\Gamma s \not \geq s$ for all $s \in \R^n_+ \backslash \{0\}$;
\item $\Gamma^k \to 0$ for $k \to \infty$;
\item the origin of the system $s(k+1) = \Gamma(s(k))$ is GAS.
\end{enumerate}
\end{remark}

For the map $\Gamma_\oplus$ defined in \eqref{eq:gainmap} we have the following equivalent condition, which gives the possibility to check the small-gain condition (see \cite[Theorem~6.4]{Ruf09Monotone}).

\begin{proposition}\label{prop:cycle}
The map $\Gamma_\oplus: \R^n_+ \rightarrow \R^n_+$ defined in \eqref{eq:gainmap} satisfies the small-gain condition \eqref{eq:sgc} if and only if all cycles in the corresponding graph of $\Gamma_\oplus$ are weakly contracting, i.e., $\gamma_{i_0i_1} \circ \gamma_{i_1i_2} \circ \ldots \circ \gamma_{i_ki_0} < \id$ for $k \in \N$, $i_j \neq i_l$ for $j \neq l$.
\end{proposition}

Note that it is sufficient that all \emph{minimal} cycles of the graph of $\Gamma_\oplus$ are weakly contracting, which means that $i_j\neq i_l$ for all $j,l \in \{0, \ldots, k\}$. Thus, $k<n$.

%%%%%%%%%%%%%%%%%%%%%%%%%%%%%%%%%%%%%%%%%%%%%%%%%%%%%%%%%%%
\section{Problem statement}\label{sec:prob}

We consider discrete-time systems of the form
\begin{equation}\label{eq:oas} x(k+1) = G(x(k),u(k)), \qquad k \in \N. \end{equation}
Here  $u(k)\in \R^m$ denotes the input at time $k \in \N$. Note that an
input is a function $u: \N \rightarrow \R^m$. By $x(k, \xi, u(\cdot)) \in
\R^n$ we denote the solution of \eqref{eq:oas}
at time $k \in \N$, starting in $x(0)= \xi \in \R^n$ with input $u(\cdot)$. 

Throughout this work the map $G:\R^n \times \R^m \rightarrow \R^n$ satisfies the following standing assumption.

\begin{assumption}\label{ass:omega}
The function $G$ in \eqref{eq:oas} is \emph{globally $\K$-bounded}, 
i.e., there exist functions $\omega_1,\omega_2\in\K$ such that for all $\xi \in \R^n$ and $ \mu \in \R^m$ we have
\begin{equation}\label{eq:omega} \dabs{G(\xi,\mu)} \leq \omega_1(\dabs{\xi}) + \omega_2(\dabs{\mu}).
\end{equation} 
\end{assumption}

Assumption~\ref{ass:omega} implies continuity of $G$ in $(0,0)$, but it does not require the map $G$ to be continuous elsewhere (as assumed e.g. in \cite{JW01-Automatica,JW02,LHJ12}) or (locally) Lipschitz (as assumed e.g. in \cite{Agarwal,AP98}). For further remarks on Assumption~\ref{ass:omega} see Remark~\ref{rem:omega} and Appendix~\ref{subsec:appendixbounds}.

\begin{definition}\label{def:ISS}
We call system~\eqref{eq:oas} \emph{input-to-state stable} if there exist $\beta\in\KL$ and $\gamma\in \K$ such that for all initial states $\xi \in \R^n$, all inputs $u(\cdot) \in l^\infty (\R^m)$ and all $k \in \N$
\begin{equation}\label{eq:ISSestimate} \dabs{x(k,\xi, u(\cdot))} \leq  \beta(\dabs{\xi},k) + \gamma(\dabsu{u}). \end{equation}
If the $\KL$-function $\beta$ in \eqref{eq:ISSestimate} can be chosen as
\begin{equation}\label{eq:expISSestimate} \beta(r,t)=C\kappa^t r \end{equation}
with $C\geq 1$ and $\kappa \in [0,1)$, then system~\eqref{eq:oas} is called exponentially input-to-state stable (expISS).
\end{definition}

An alternative definition of ISS replaces the sum in \eqref{eq:ISSestimate} by the maximum. 
Indeed, both definitions are equivalent, and the equivalence even holds for more general definitions of ISS using \emph{monotone aggregation functions}, see \cite[Proposition~2.5]{GWmixed13}.

\begin{remark}\label{rem:omega} Since we are interested in checking the ISS property of system~\eqref{eq:oas}, it is clear that the existence of functions $\omega_1, \omega_2 \in \K$ satisfying \eqref{eq:omega} in Assumption~\ref{ass:omega} is no restriction, since every ISS system necessarily satisfies \eqref{eq:omega}. 
In particular, by~\eqref{eq:ISSestimate}, we have
\begin{equation*}
\dabs{G(\xi,\mu)} = \dabs{x(1,\xi,\mu)} \leq \beta(\dabs{\xi},1)+ \gamma(\dabs{\mu})
\end{equation*}
and we may choose $\omega_1(\cdot) = \beta(\cdot, 1)$ and  $\omega_2(\cdot) = \gamma(\cdot)$ to obtain~\eqref{eq:omega}. 
Moreover, for expISS systems we can take $\omega_1(s)=C\kappa s$
where $C \geq 1$ and $\kappa \in [0,1)$ stem from~\eqref{eq:expISSestimate}. 
In other words, any expISS system is globally $\K$-bounded with a linear function $\omega_1 \in \K$.
\end{remark}

The following lemma shows that by a suitable \emph{change of coordinates}, i.e., a homeomorphism $T:\R^n\rightarrow \R^n$ with $T(0)=0$ (see e.g. \cite{KD13-AuCC}),  we can always assume that $\omega_1 \in \K$ in \eqref{eq:omega} is linear.

\begin{lemma}\label{lem:coordinate} 
Consider system~\eqref{eq:oas} and let Assumption~\ref{ass:omega} hold. 
Then there exists a change of coordinates $T$ such that for $z(k):= T(x(k))$ the system
\begin{equation}\label{eq:system-transformed}
z(k+1) = \tilde G(z(k),u(k)), \qquad \forall k\in \N
\end{equation}
satisfies~\eqref{eq:omega} with linear $\omega_1 \in \K$.
\end{lemma}

\begin{proof} Consider a change of coordinates $T: \R^n \rightarrow \R^n$, and define $z(k):= T(x(k))$, where $x(k)$ comes from \eqref{eq:oas}. Then $z$ satisfies \eqref{eq:system-transformed} with
\begin{equation*}\tilde G(z,u) = T(G(T^{-1}(z),u)).
\end{equation*} 
Note that $\tilde G(0,0)=0$ since $T$ and its inverse fix the origin.
Furthermore, let $\omega_1, \omega_2 \in \K$ satisfy \eqref{eq:omega} for the map $G$. 
Without loss of generality, we assume that $(2 \omega_1- \id) \in \Ki$, else increase $\omega_1$. Take any $\lambda>1$. 
By \cite[Lemma~19]{KT04} there exists a $\Ki$-function $\varphi$ satisfying
\begin{equation}\label{eq:varphi-coordinate}
\varphi ( 2\omega_1(s)) = \lambda \varphi(s) \qquad \forall s \geq 0.
\end{equation}
Define $T(x) := \varphi(\dabs{x})\frac{x}{\dabs{x}}$ for $x\neq 0$, and $T(0)=0$. 
Clearly, $T$ is continuous for $x \neq 0$. 
On the other hand, $\dabs{T(x)} = \varphi(\dabs{x})$, so continuity of $T$ in zero is implied by continuity of $\varphi$ and $\varphi(0)=0$.
With $z=T(x)$ a direct computation, using again that $\dabs{z} =
\varphi(\dabs{x})$ and that $x$ and $z$ are on the same ray,
yields $T^{-1}(z) := \varphi^{-1}(\dabs{z})\frac{z}{\dabs{z}}$ for $z\neq 0$ and $T^{-1}(0)=0$. 
By the same arguments as above also $T^{-1}$ is continuous. Hence, $T$ is a homeomorphism.
Moreover, we obtain the following estimate
\begin{align*}
\dabs{\tilde G(\tilde \xi, \tilde  \mu)} & = \varphi \left( \dabs{G\left(\varphi^{-1}(\dabs{\tilde \xi})\tfrac{\tilde \xi}{\dabs{\tilde \xi}}, \tilde \mu\right)} \right) 
\stackrel{\eqref{eq:omega}}{\leq} \varphi \left( \omega_1(\varphi^{-1}(\dabs{\tilde \xi})) + \omega_2(\dabs{\tilde \mu}) \right) \\
& \leq \varphi \left( 2\omega_1(\varphi^{-1}(\dabs{\tilde \xi}))  \right) + \varphi\left(2 \omega_2(\dabs{\tilde \mu}) \right) 
\stackrel{\eqref{eq:varphi-coordinate}}{=} \lambda\dabs{\tilde \xi} + \varphi\left(2\omega_2(\dabs{\tilde \mu}) \right).
\end{align*}
So, $\tilde G$ satisfies \eqref{eq:omega} 
with the linear function $\omega_1(s)=\lambda s$, $s \in \R_+$, which concludes the proof.
\hfill~\end{proof}

\begin{remark}\label{rem:0GAS}
We further note that ISS implies \emph{global asymptotic stability of the origin with $0$ input} ($0$-GAS), i.e., the existence of a class-$\KL$ function $\beta$ such that for all $\xi \in \R^n$ and all $k \in \N$,
\begin{equation*}
\dabs{x(k,\xi, 0)} \leq  \beta(\dabs{\xi},k). \end{equation*}
In \cite{Angeli99} the author shows that for discrete-time systems~\eqref{eq:oas} with continuous dynamics \emph{integral input-to-state stability} (iISS) is equivalent to 0-GAS. Note that this is not true in continuous time.
\end{remark}

To prove ISS of system~\eqref{eq:oas} the concept of ISS Lyapunov functions is widely used (see e.g. \cite{JW01-Automatica}). 
Note that in the following definition we do not require continuity of the ISS Lyapunov function.
To shorten notation, we call a function $W:\R^n \rightarrow \R_+$ \emph{proper and positive definite} if
there exist $\alpha_1, \alpha_2 \in \Ki$ such that for all $\xi \in \R^n$
\begin{equation*} \alpha_1 (\dabs{\xi}) \leq W(\xi) \leq \alpha_2(\dabs{\xi}). \end{equation*}
\begin{definition}\label{def:ISSLF}
A proper and positive
 definite
 function $W:\R^n \rightarrow \R_+$ is called
a \emph{dissipative ISS Lyapunov function} for system~\eqref{eq:oas}
if there exist $\sigma \in \K$ and a positive definite function $\rho$ with $(\id - \rho) \in \Ki$ such that for any $\xi \in \R^n, \mu \in \R^m$
\begin{equation} \label{eq:ISSdissipative}W(G(\xi,\mu)) \leq \rho(W(\xi)) + \sigma(\dabs{\mu}).\end{equation}
\end{definition}

\begin{remark}\label{rem:-alpha3}
(i)
In many prior works (e.g. \cite{JW01-Automatica,LHJ12}) the definition of a dissipative ISS Lyapunov function requires the existence of a function $\alpha_3\in \Ki$ and a function $\sigma\in\K$ such that 
\begin{equation} \label{eq:dissipative2}
W(G(\xi, \mu))-W(\xi) \leq -\alpha_3(\dabs{\xi}) + \sigma(\dabs{\mu}) 
\end{equation} 
holds for all $\xi \in \R^n, \mu \in \R^m$. Let us briefly explain, that this requirement is equivalent to Definition~\ref{def:ISSLF}. 
First, from~\eqref{eq:dissipative2} and the positive definiteness of $W$ we get
$0 \leq W (G(\xi, \mu)) \leq W (\xi) - \alpha_3(\dabs{\xi}) + \sigma(\dabs{\mu}) \leq (\id - \alpha_3 \circ \alpha_2^{-1})(W(\xi)) +  \sigma(\dabs{\mu}) = \rho(W(\xi)) +  \sigma(\dabs{\mu})$ with $\rho := (\id -  \alpha_3 \circ \alpha_2^{-1})$ positive definite, and $(\id-\rho)=\alpha_3\circ \alpha_2^{-1}\in \Ki$. 
So~\eqref{eq:dissipative2} implies~\eqref{eq:ISSdissipative}.
To show the other implication, note that since $0 \leq W(G(\xi,\mu))  \leq (\alpha_2- \alpha_3)(\dabs{\xi})+ \sigma(\dabs{\mu})$ holds it follows that $ \alpha_2(s) \geq \alpha_3(s)$ for all $s \in \R_+$  by taking $\mu=0$.
Let~\eqref{eq:ISSdissipative} hold with positive definite function $\rho$ satisfying $(\id - \rho) \in \Ki$, then we get $ W(G(\xi,\mu))-W(\xi) \leq -\alpha_3(\dabs{\xi})+ \sigma(\dabs{\mu})$ for $\alpha_3:=(\id-\rho)\circ \alpha_1\in~\Ki$, which is~\eqref{eq:dissipative2}.
\medskip
\\(ii)
For systems with external inputs $u$ there are usually two forms of ISS
Lyapunov functions.
 The first one is the dissipative form of Definition~\ref{def:ISSLF}.
The other 
type are frequently called implication-form ISS Lyapunov functions. These are proper and positive definite function $W:\R^n \rightarrow \R_+$ satisfying 
\begin{equation}\label{eq:ISSimplication} \dabs{\xi} \geq \chi(\dabs{\mu})\quad  \Rightarrow \quad W(G(\xi,\mu))\leq \bar \rho(W(\xi)) .\end{equation}
for all $\xi \in \R^n$, $\mu \in \R^m$, and some
positive definite function $\bar \rho<\id$ and $\chi\in \K$.

If the function $G$ in~\eqref{eq:oas} is \emph{continuous} then conditions~\eqref{eq:ISSdissipative} and~\eqref{eq:ISSimplication} are equivalent, see~\cite[Remark~3.3]{JW01-Automatica} and~\cite[Proposition~3.3 and 3.6]{GK-discontinuousISS}.
So the existence of a dissipative or implication-form ISS Lyapunov function implies ISS of the system if the dynamics are continuous.

If $G$ is \emph{discontinuous} then the equivalence between the existence of dissipative and implication-form ISS Lyapunov functions is no longer satisfied.
Indeed, any dissipative ISS Lyapunov function is an implication-form ISS Lyapunov function, 
but the converse does not hold in general, see~\cite{GK-discontinuousISS}. 
In particular, for discontinuous dynamics, an implication-form ISS Lyapunov function is not sufficient to conclude ISS, see also~\cite[Remark~2.1]{LailaNesic03} and \cite[Example~3.7]{GK-discontinuousISS}.
\medskip
\\(iii)
To prove ISS of system~\eqref{eq:oas}, the authors in~\cite[Proposition~2.4]{GK-discontinuousISS} have shown that
the assumption $(\id - \rho) \in \Ki$ in Definition~\ref{def:ISSLF} can be weakened to the condition $(\id-\rho) \in \K$ and $\sup (\id-\rho)>\sup \sigma$.
Moreover, for any ISS system~\eqref{eq:oas} there exists a dissipative ISS Lyapunov function $W$ with linear decrease function $\rho$, see \cite[Theorem~2.6]{GK-discontinuousISS}.
\end{remark}

We relax the condition~\eqref{eq:ISSdissipative} in Definition~\ref{def:ISSLF} by replacing the solution after one time step $G(\xi, \mu) = x(1, \xi, \mu)$ by the solution after a finite number of time steps. 
This relaxation was recently introduced in \cite{LDA13,GLW13-SG,GGLW13-SCL} for systems without inputs. In the context of ISS, i.e., for systems with inputs, this concept appears to be new.
 
\begin{definition}\label{def:ftISSLF}
A proper and positive definite function $V:\R^n \rightarrow \R_+$ is called a \emph{dissipative finite-step ISS Lyapunov function} for system~\eqref{eq:oas} if there exist an $M\in \N$, $\sigma \in \K$, a positive definite function $\rho$ with $(\id-\rho)\in \Ki$ such that for any $\xi \in \R^n$, $u(\cdot) \in l^\infty (\R^m)$
\begin{equation} 
\label{eq:decayV}
V(x(M,\xi,u(\cdot))) \leq  \rho(V(\xi)) + \sigma(\dabsu{u}).  \end{equation}
\end{definition}

At first glance the task of finding an ISS Lyapunov function appears to have become easier, as we are only required to satisfy a condition after a finite number of steps. 
We will show later that there is some truth to this point of view, in that it is possible to show that a simple class of functions always yields a dissipative finite-step ISS Lyapunov function. 
Unfortunately, there is now a new general question: It is not sufficient to know a dissipative finite-step ISS Lyapunov function, but we also require to know the constant $M$, which may be hard to characterize.

In the next section we study properties of  dissipative finite-step ISS Lyapunov functions.

% % % % % % % % % % % % % % % % % % % % % % % % % % % % % % % % % % % % % % % % % % % % % 
\section{Dissipative Finite-Step ISS Lyapunov Theorems}\label{sec:ftISS}

We start this section by proving that the existence of a dissipative finite-step ISS Lyapunov function is sufficient to conclude ISS of system~\eqref{eq:oas}. 
As any dissipative ISS Lyapunov function is a particular dissipative finite-step ISS Lyapunov function, this result is closely related to~\cite[Lemma~3.5]{JW01-Automatica}. Furthermore, the class of dissipative ISS Lyapunov functions is a strict subset of the class of dissipative finite-step ISS Lyapunov functions. Hence, this result is more general than showing that the existence of a dissipative ISS Lyapunov function implies ISS of the underlying system.
The proof requires a comparison lemma and an additional lemma, which are given in the appendix.

\begin{theorem}\label{theo:ftLF-ISS} If there exists a dissipative finite-step ISS Lyapunov function for system~\eqref{eq:oas} then system~\eqref{eq:oas} is ISS.
\end{theorem}

The proof follows the lines of \cite[Lemma~3.5]{JW01-Automatica}, which establishes that for continuous dynamics the existence of a continuous dissipative ISS Lyapunov function implies ISS of the system. 
Note that in this work we consider global $\K$-boundedness of the system dynamics instead of continuity such as in~\cite{JW01-Automatica}. 

\begin{proof}
Let $V$ be a dissipative finite-step ISS Lyapunov function satisfying Definition~\ref{def:ftISSLF} for system \eqref{eq:oas} with suitable $\alpha_1, \alpha_2 \in \Ki$, $M \in \N$, $\sigma\in \K$, and a positive definite function $\rho$ with  $(\id-\rho)\in\Ki$. 
Let $\xi \in \R^n$ and fix any input $u(\cdot)\in l^\infty (\R^m)$. 
We abbreviate the state $x(k):=x(k, \xi, u(\cdot))$. 
Let $\nu\in \Ki$ be such that $\id-\nu \in \Ki$ and consider the set 
\begin{equation*} \Delta:= \left \{ \xi \in \R^n : V(\xi) \leq \delta :=(\id- \rho)^{-1} \circ \nu^{-1} \circ \sigma(\dabsu{u}) \right \}.
\end{equation*}
 We will now show that for any $k \in \N$ with $x(k) \in \Delta$ we have $x(k+lM) \in \Delta$ for all $l\in \N$. Using~\eqref{eq:decayV}, a direct computation yields
 \begin{align*}
 V(x(k+M)) & \leq \rho(V(x(k))) + \sigma(\dabsu{u}) 
  \leq \rho(\delta) + \sigma(\dabsu{u}) \\
 & =  - (\id -\nu) \circ (\id-\rho) (\delta) + \delta -  \nu \circ (\id-\rho) (\delta) + \sigma(\dabsu{u}) \\
 &= - (\id -\nu) \circ (\id-\rho) (\delta) + \delta  \leq \delta.
 \end{align*}
Hence, $x(k+M) \in \Delta$ and by induction we get $x(k+lM) \in \Delta$ for all $l \in \N$.

Let $j_0 \in \N \cup \{\infty\}$ satisfy $j_0 := \min \{ k \in \N: x(k), \ldots, x(k+M-1) \in \Delta \}$. 
By definition of $j_0$ and by the above consideration, we see that $x(k) \in \Delta$ for all $k \geq j_0$. Thus, we have
\begin{equation} \label{eq:V-leq-gamma}
 V(x(k)) \leq  (\id- \rho)^{-1} \circ \nu^{-1} \circ \sigma (\dabsu{u}) =: \tilde \gamma (\dabsu{u}).
\end{equation}
For $k < j_0$, we have to consider two cases. 

First, if
$x(k) \in \Delta$ then by definition of $\Delta$ we have $V(x(k)) \leq \tilde \gamma (\dabsu{u})$.
Secondly, if
$x(k)  \not \in \Delta$ 
let $l \in \N$ and $k_0 \in \{0, \ldots, M-1\}$ satisfy $k = lM+k_0$.
Since $x(k_0) \in \Delta$ implies $x(lM+k_0) \in \Delta$, we conclude $x(k_0) \not \in \Delta$.
Hence, by definition of $\Delta$,
 $V(x(k_0)) > (\id- \rho)^{-1} \circ \nu^{-1} \circ \sigma(\dabsu{u})$, or, equivalently, 
$\sigma(\dabsu{u}) <  \nu \circ (\id - \rho) \circ V(x(k_0))$,
 which  implies 
\begin{align*}
V(x(k_0+M)) & \leq \rho(V(x(k_0))) + \sigma(\dabsu{u}) \\
&< \rho(V(x(k_0))) + \nu \circ (\id - \rho) \circ V(x(k_0)) \\
 &= (\rho + \nu  \circ (\id-  \rho))\circ V(x(k_0)).
\end{align*}
Note that the function $\chi:=(\rho + \nu\circ (\id -  \rho)) $ satisfies $\chi= \id - (\id - \nu) \circ (\id-\rho)<\id$. 
Let $L:= \sup \{l \in \N \, : \, V(x(lM+k_0)) \not \in \Delta \}$.
Then we have for all $l \in \{0, \ldots, L\}$
\begin{equation*}
V(x((l+1)M+k_0)) \leq \chi (V(x(lM+k_0))).
\end{equation*}
Note that the function $\chi = \rho + \nu \circ (\id-\rho)$ is continuous, positive definite and  unbounded as $\nu, (\id-\rho) \in \Ki$, and it satisfies $\chi(0)=0$. 
Hence, we can without loss of generality assume that $\chi \in \Ki$, else pick $\tilde \chi \in \Ki$ satisfying $\chi \leq \tilde \chi < \id$.
Applying Lemma~\ref{lem:comparison}
there exists a $\KL$-function $\beta_{k_0}$ satisfying
\begin{equation*}
V(x(lM+k_0)) \leq  \beta_{k_0}( V(x(k_0)), lM+k_0)
\end{equation*}
for all $l \in \{0, \ldots, L\}$.
Moreover, for all $l>L$, we have $V(x(lM+k_0)) \in \Delta$ implying $V(x(lM+k_0)) \leq \tilde \gamma (\dabsu{u})$.
Thus, for all $l \in \N$, we have
\begin{equation}\label{eq:VxlM+k_0}
V(x(lM+k_0)) \leq \max \left\{  \beta_{k_0} \left(V(x(k_0)), lM+k_0 \right),  \tilde \gamma (\dabsu{u}) \right\}.
\end{equation}
It is important to note that both $\tilde \gamma$ and $\chi$ are independent on the choice of $\xi \in \R^n$ and $u(\cdot) \in l^\infty (\R^m)$.
In addition, by the proof of Lemma~\ref{lem:comparison}, also $\beta_{k_0} \in \KL$ does not depend on $\xi \in \R^n$ and $u(\cdot) \in l^\infty(\R^m)$. 
Hence,~\eqref{eq:VxlM+k_0} holds for all solutions $x(k)$.

Define the $\KL$-function
\begin{equation*}
 \tilde \beta(s,r) := 
 \max_{k_0 \in \{0, \ldots, M-1\}} 
 \beta_{k_0} (s,r)
\end{equation*}
and $V_M^{\max}(\xi, u(\cdot)):= \max_{j \in \{0, \ldots, M-1\}} V(x(j, \xi, u(\cdot)))$.
Then for all $k \in \N$, all $\xi \in \R^n$ and all $u(\cdot) \in l^\infty(\R^m)$ we have
\begin{equation*}
V(x(k)) \leq \max \left\{ \tilde \beta (V_M^{\max} (\xi, u(\cdot)),k), \tilde \gamma (\dabsu{u}) \right\}.
\end{equation*}
Consider $\vartheta_j, \zeta_j \in \K$  from 
Lemma~\ref{lem:bounds} and define
$\tilde \vartheta := \max_{j\in\{0, \ldots, M-1\}} \alpha_2(2\vartheta_j)$ and $\tilde \zeta := \max_{j\in\{0, \ldots, M-1\}} \alpha_2(2\zeta_j)$.
Then for all $\xi \in \R^n$ and $u(\cdot) \in l^\infty(\R^m)$
we get
\begin{equation*}
V_M^{\max} (\xi, u(\cdot)) \leq \max_{j \in \{0, \ldots, M-1\}} \alpha_2 (\dabs{x(j)}) \leq \tilde \vartheta(\dabs{\xi}) + \tilde \zeta (\dabsu{u}).
\end{equation*}
So all in all we have for all $k \in \N$, all $\xi \in \R^n$ and all $u(\cdot) \in l^\infty$,
\begin{align*}
V(x(k)) &\leq \max \left\{ \tilde  \beta(\tilde \vartheta(\dabs{\xi}) + \tilde \zeta (\dabsu{u}), k), \tilde \gamma(\dabsu{u}) \right\} \\
 &\leq \max \left\{ \tilde  \beta(2\tilde \vartheta(\dabs{\xi}), k) + \tilde \beta(2\tilde \zeta (\dabsu{u}),0), \tilde \gamma(\dabsu{u}) \right\} \\
&\leq  \tilde  \beta(2\tilde \vartheta(\dabs{\xi}), k) +  \left (\tilde \beta(2\tilde \zeta (\dabsu{u}),0) + \tilde \gamma(\dabsu{u}) \right).
\end{align*}
Hence, we get \eqref{eq:ISSestimate} by defining 
$\beta(s,r) := \alpha_1^{-1}(2\tilde \beta(2\tilde \vartheta(s),r))$ 
and \linebreak%
$\gamma(s) := \alpha_1^{-1}\left(2 (\tilde \beta(2\tilde \zeta (\dabsu{u}),0) + \tilde \gamma(\dabsu{u}) )\right)$. 
Note that for fixed $r \geq 0$, $\beta(\cdot, r)$ is a $\K$-function as the composition of $\K$-functions, and for fixed $s >0$, $\beta(s, \cdot) \in \L$, since the composition of $\K$- and $\L$-functions is of class $\L$ (see \cite[Section~24]{Hahn67}, \cite[Section~2]{Kellet-comparison}), so really $\beta \in \KL$. 
Further note that the summation of class-$\K$ functions yields a class-$\K$ function, so $\gamma \in \K$.
\hfill~\end{proof}

\begin{remark}
To clarify the concept of dissipative finite-step ISS Lyapunov functions we now discuss the connection to higher order iterates of system~\eqref{eq:oas}.

Let $G:\R^n\times \R^m \to \R^n$ from~\eqref{eq:oas} be given. Then, for any $i\in \N$ with $i\geq 1$, we define the $i$th iterate of $G$, denoted by $G^i : \R^n \times (\R^m)^i \to \R^n $, as follows:
\begin{align*}
\xi\in \R^n,\,w_1=u_1 \in\R^m &\,\mapsto\, G^1(\xi,w_1):= G(\xi,u_1), 
\\ \xi \in \R^n, \,w_i:=(u_1,\ldots,u_i) \in  (\R^m)^i & \,\mapsto \, G^i(\xi,w_i):= G(G^{i-1}(\xi,w_{i-1}),u_i)
\end{align*}
with $i\in \N$, $i\geq 2$.
Now fix any $M\in\N$, and consider the system
\begin{equation}\label{eq:oasMiterate}
\bar x(k+1) = G^M(\bar x(k),w_M(k))
\end{equation}
with state $\bar x \in \R^n$ and input function $w_M(\cdot)=(u_1(\cdot),\ldots, u_M(\cdot))$ taking values in $(\R^m)^M$. 
Firstly, for any $k\in\N$ there exist unique $l\in\N$ and $i\in \{1,\ldots,M\}$ such that $k=lM+i$. 
For $u: \N \to \R^m$ define $u_i$, $i \in \{1,..,M\}$, by
\begin{equation*}
u_i(l):= u(lM+i), \qquad l\in\N,
\end{equation*}
we call~\eqref{eq:oasMiterate} the 
$M$-iteration corresponding to system~\eqref{eq:oas}. 
Note that $\dabsu{w_M}\!:=\max\{ \dabsu{u_1},\ldots, \dabsu{u_M}\}=\dabsu{u}$.
It is not difficult to see that for all $j\in \N$ and all $\xi \in \R^n$ we have
\begin{equation}\label{eq:xbarx-iterated}
x(jM,\xi,u(\cdot)) = \bar x(j,\xi,w_M(\cdot)).
\end{equation}
Thus, if system~\eqref{eq:oas} is ISS, i.e., it satisfies~\eqref{eq:ISSestimate}, then also the $M$-iteration~\eqref{eq:oasMiterate} is ISS and satisfies
\begin{equation*}
\dabs{\bar x(j,\xi,w_M(\cdot))} \stackrel{\eqref{eq:xbarx-iterated}}{=} \dabs{x(jM,\xi,u(\cdot))} \leq \beta(\dabs{\xi},jM)+\gamma(\dabsu{u}) =: \bar \beta(\dabs{\xi},j)+\gamma(\dabsu{w_M}).
\end{equation*}
Moreover, a dissipative finite-step ISS Lyapunov function for system~\eqref{eq:oas} with suitable $M\in \N$ is also a dissipative ISS Lyapunov function for the $M$-iteration~\eqref{eq:oasMiterate}.

Conversely, let system~\eqref{eq:oasMiterate} be ISS then there exists a dissipative ISS Lyapunov function $V$ for system~\eqref{eq:oasMiterate} (see e.g. \cite[Theorem~1]{JW01-Automatica} for continuous $G^M$ or \cite[Lemma~2.3]{GK-discontinuousISS} for discontinuous $G^M$). 
From~\eqref{eq:xbarx-iterated} we see that $V$ is also a dissipative finite-step ISS Lyapunov function for system~\eqref{eq:oas}, and by Theorem~\ref{theo:ftLF-ISS} we conclude that system~\eqref{eq:oas} is ISS.
\end{remark}

Summarizing, we obtain the following corollary.

\begin{corollary}
System~\eqref{eq:oas} is ISS if and only if the 
$M$-iteration~\eqref{eq:oasMiterate} is ISS. In particular, a function $V:\R^n \rightarrow \R_+$ is a dissipative Lyapunov function for system~\eqref{eq:oasMiterate} if and only if it is a dissipative finite-step Lyapunov function for system~\eqref{eq:oas}.
\end{corollary}

As finding a (dissipative) ISS Lyapunov function is a hard task, we will see in the remainder of this work that finding a dissipative finite-step ISS Lyapunov function (or, equivalently, a dissipative Lyapunov function for a corresponding $M$-iteration) is sometimes easier. Furthermore, if we impose stronger conditions on the dissipative finite-step ISS Lyapunov function and the dynamics, then we can conclude an exponential decay of the bound on the system's state.

\begin{theorem}\label{theo:ftLF-expISS} 
Let system~\eqref{eq:oas} be globally $\K$-bounded with linear $\omega_1 \in \Ki$.
If there exists a dissipative finite-step ISS Lyapunov function $V$ for system~\eqref{eq:oas} satisfying for any $\xi \in \R^n$ and $u(\cdot) \in l^\infty (\R^m)$
\begin{align*}
 a\dabs{\xi}^\lambda   \leq V(\xi) \leq b \dabs{\xi}^\lambda , \\
 V(x(M,\xi, u(\cdot)))  \leq c V(\xi) + d \dabsu{u}
\end{align*}
with $b\geq a>0, c \in [0,1)$ and $d,\lambda>0$, then system~\eqref{eq:oas} is expISS.
\end{theorem}

\begin{proof}
The proof follows the lines of the proof of Theorem~\ref{theo:ftLF-ISS}. Hence, we will omit the detailed proof, and only give a sketch. 

Consider the linear global $\K$-bound $\omega_1 \in \Ki$.
We assume that  $\omega_2\in \K$ is a linear function, too. 
This assumption is only for  simplifying the proof, but does not change the result.
First note that in the proof of Theorem~\ref{theo:ftLF-ISS} we can choose $\nu(s) = hs$ with $h \in (0,1)$, and since $\rho$ and $\sigma$ are linear $\Ki$-functions, we obtain that
$\tilde \gamma (s):=  (\id- \rho)^{-1} \circ \nu^{-1} \circ \sigma(s) = \tfrac{d}{h(1-c)}s$ is a linear function. Furthermore, in the case that $x(k) \not \in \Delta$, we see that 
\begin{equation*} V(x(k+M, \xi, u(\cdot))) \leq (c+h(1-c)) V(x(k, \xi, u(\cdot))).
\end{equation*}
Define $\tilde \kappa:=(c+h(1-c))<1$. In this case using the comparison Lemma~\ref{lem:comparisonexp} we obtain the estimate
\begin{equation*}
V(x(k, \xi, u(\cdot))) \leq  \tfrac{\left(\tilde \kappa^{1/M} \right)^k}{\tilde \kappa}  V_M^{\max}(\xi, u(\cdot)),
\end{equation*}
where $ V_M^{\max}(\xi, u(\cdot)) := \max_{j \in \{0, \ldots, M-1\}} V(x(j, \xi, u(\cdot)))$. 
Let $\omega_1(s):= w_1s$ and $\omega_2(s):= w_2 s$ for $s \in \R_+$ and $w_1,w_2 >0$.
Using Lemma~\ref{lem:boundsexp}, the estimate~\eqref{eq:bounds} is satisfied for $\vartheta_j (s) =  w_1^js$ and $\zeta_j(s)= w_2 \sum_{i=0}^{j-1}w_1^{i} s$.
Thus, 
\begin{equation*}
V(x(j, \xi, u(\cdot))) \!\leq\!  b \left(\dabs{x(j, \xi, u(\cdot))} \right)^\lambda \!\leq\! b \!\left(\!\!w_1^j \dabs{\xi} \!+\! w_2 \sum_{i=0}^{j-1}w_1^{i} \dabsu{u}\!\! \right)^\lambda \!\!\!=\!  b\left(\tilde w_1\dabs{\xi} \!+\! \tilde w_2\dabsu{u} \right)^\lambda
\end{equation*}
with $\tilde w_1 := \max_{j \in \{0, \ldots, M-1\}}w_1^j $, and $\tilde w_2:= \max_{j \in \{0, \ldots, M-1\}} w_2 \sum_{i=0}^{j-1}w_1^{i} $, and hence,  
\begin{align*}
V(x(k, \xi, u(\cdot)))
 & \leq \max_{j \in \{0, \ldots, M-1\}}  \frac{\tilde \kappa^{k/M}}{\tilde \kappa} b \left(w_1^j \dabs{\xi} + w_2 \sum_{i=0}^{j-1}w_1^{i} \dabsu{u} \right)^\lambda\\
 & \leq   \tfrac{b}{\tilde \kappa} \tilde \kappa^{k/M}  \left(\tilde w_1\dabs{\xi} + \tilde w_2\dabsu{u} \right)^\lambda.
\end{align*}
This implies that for all $\xi \in \R^n$ and all $u(\cdot) \in l^\infty (\R^m)$ we have
\begin{align*}
\dabs{x(k, \xi, u(\cdot))}
\leq  \left( a^{-1} V(x(k, \xi, u(\cdot))) \right)^{1/\lambda} 
& \leq \left( \tfrac{b}{a\tilde \kappa} \tilde \kappa^{k/M} \right)^{1/\lambda} \left(\tilde w_1\dabs{\xi} + \tilde w_2\dabsu{u} \right)\\
& \leq \left( \frac{b \tilde \omega_1^{1/\lambda}}{a\tilde \kappa} \right)^\lambda \kappa^k \dabs{\xi} +  \left( \frac{b \tilde \omega_2^{1/\lambda}}{a\tilde \kappa} \right)^\lambda \dabsu{u},
\end{align*}
with $\kappa:= \tilde \kappa^{1/\lambda M} < 1$. 
So, system~\eqref{eq:oas} satisfies \eqref{eq:ISSestimate} with $\beta$ as in \eqref{eq:expISSestimate}, where $C= \left( \frac{b \tilde \omega_1^{1/\lambda}}{a\tilde \kappa} \right)^\lambda $ and $\kappa <1$ as defined above. 
Hence, system~\eqref{eq:oas} is expISS.
\hfill~\end{proof}

While Theorem~\ref{theo:ftLF-ISS} shows the sufficiency of the existence of dissipative finite-step ISS Lyapunov functions to conclude ISS of system~\eqref{eq:oas}, we are now interested in the necessity. 
At this stage we can exploit the fact that any dissipative ISS Lyapunov function as defined in Definition~\ref{def:ISSLF} is a particular dissipative finite-step ISS Lyapunov function satisfying Definition~\ref{def:ftISSLF} with $M=1$. 

\begin{proposition}\label{prop:ISS->ftLF} If system~\eqref{eq:oas} is ISS then there exists a dissipative finite-step ISS Lyapunov function for system~\eqref{eq:oas}.
\end{proposition}

\begin{proof}
If the right-hand side $G: \R^n \times \R^m \rightarrow \R^n$ of system~\eqref{eq:oas} is continuous, then \cite[Theorem~1]{JW01-Automatica} implies the existence of a smooth function $V:\R^n\rightarrow \R_+$ satisfying 
$\alpha_1(\dabs{\xi}) \leq V(\xi) \leq \alpha_2 (\dabs{\xi})$
and
$V(G(\xi, \mu))-V(\xi) \leq - \alpha_3(\dabs{\xi})+\sigma(\dabs{\mu})$
for all $\xi \in \R^n$, $\mu \in \R^m$, and suitable $\alpha_1, \alpha_2, \alpha_3 \in \Ki$, $\sigma \in \K$. Then with Remark~\ref{rem:-alpha3} and $M=1$ it is easy to see that $V$ is a dissipative finite-step ISS Lyapunov function. 
This result also applies to discontinuous dynamics, see \cite[Lemma~2.3]{GK-discontinuousISS}.
\hfill~\end{proof}

Obviously, Proposition~\ref{prop:ISS->ftLF} makes use of the converse ISS
Lyapunov theorem in \cite{JW01-Automatica,GK-discontinuousISS} to
guarantee the existence of a dissipative (finite-step) ISS Lyapunov
function. 

For the case of expISS systems of the form~\eqref{eq:oas}, it turns out that norms are dissipative finite-step ISS Lyapunov functions.

\begin{theorem} \label{theo:converseexpISS}
If system~\eqref{eq:oas} is expISS then the function $V:\R^n \!\to\! \R_+$ defined by
\begin{equation}\label{eq:condV}
V(\xi):= \dabs{\xi}, \qquad \xi \in \R^n
\end{equation}
 is a dissipative finite-step ISS Lyapunov function for system~\eqref{eq:oas}.
\end{theorem}

\begin{proof}
System~\eqref{eq:oas} is expISS if it satisfies \eqref{eq:ISSestimate} with \eqref{eq:expISSestimate} for constants $C\geq 1$ and $\kappa \in [0,1)$. Take $M\in \N$ such that $C\kappa^M<1$, and $V$ as defined in \eqref{eq:condV}. 
Clearly, $V$ is proper and positive definite with $\alpha_1=\alpha_2 = \id \in \Ki$.
On the other hand, for any $\xi \in \R^n$, we have
\begin{align*}
V(x(M,\xi,u(\cdot))) &=  \dabs{x(M,\xi,u(\cdot))} 
\leq   C\kappa^M\dabs{\xi} + \gamma(\dabsu{u}) \\
& =  C\kappa^{M}V(\xi) + \gamma(\dabsu{u})  
=: \rho(V(\xi)) +  \sigma(\dabsu{u})
\end{align*}
where $ \rho(s) :=C\kappa^Ms<s$ for all $s>0$, since $C\kappa^M<1$. Note that $(\id-\rho)(s) = (1-C\kappa^M)s \in \Ki $  and $\sigma:= \gamma \in \Ki$,
which shows~\eqref{eq:decayV}. So $V$ defined in~\eqref{eq:condV} is a dissipative finite-step ISS Lyapunov function for system~\eqref{eq:oas}. 
\hfill~\end{proof}

We emphasize that the hard task in Theorem~\ref{theo:converseexpISS} is finding a sufficiently large $M\in \N$. 
However, Theorem~\ref{theo:converseexpISS} 
suggests to take the norm as a candidate for a dissipative finite-step ISS Lyapunov function. Verification for this candidate function to be a dissipative finite-step ISS Lyapunov function can be done as outlined in the following procedure.

\begin{proc}\label{proc:LFdirect}
Consider system~\eqref{eq:oas} and assume that Assumption~\ref{ass:omega} holds.
\begin{enumerate}
\item[{[1]}] Set $k=1$.
\item[{[2]}] Check 
\begin{equation*}
\dabs{x(k,\xi,u(\cdot))} \leq c \dabs{\xi}+ \sigma(\dabs{u}_{[0,k]})
\end{equation*} 
for all $\xi \in \R^n$, $u(\cdot) \in l^\infty (\R^m)$ with suitable $c \in [0,1)$ and $\sigma \in \Ki$. 
If the inequality holds set $M=k$; else set $k=k+1$ and repeat.
\end{enumerate}
If this procedure is successful, then $V(\xi):=\dabs{\xi}$ is a dissipative finite-step ISS Lyapunov function for system~\eqref{eq:oas}.
By Theorem~\ref{theo:ftLF-expISS} system~\eqref{eq:oas} is expISS. 
\end{proc}

At this point we should mention that the procedure just described is far
from being an algorithm. The question how to really perform item [2],
which should result in global bounds, would
depend on the existence of suitable analytic or numerical bounds which may
or may not exist depending on the system class at hand. For easy examples
this can surely be done by hand. How to address this question
systematically would depend on specific situations and we will not discuss
this issue here. It is clear that in general this is a very hard task.

% % % % % % % % % % % % % % % % % % % % % % % % % % % % % % % % % % % % % % % % % % % 
\section{Relaxed ISS Small-Gain Theorems} \label{sec:sgt}
% % % % % % % % % % % % % % % % % % % % % % % % % % % % % % % % % % % % % % % % % % % 

In this section we consider system~\eqref{eq:oas} split into $N$ subsystems of the form
\begin{equation}\label{eq:subsystem}
x_i(k+1)= g_i(x_1(k), \ldots, x_N(k), u(k)), \qquad k \in \N,
\end{equation}
with $x_i(0) \in \R^{n_i}$ and $g_i:\R^{n_1} \times \ldots \times \R^{n_N} \times \R^m \rightarrow \R^{n_i}$ for $i \in \{1, \ldots, N\}$. We further let $n= \sum_{i=1}^N n_i$, $x=(x_1, \ldots,x_N) \in \R^n$, then with $G=(g_1, \ldots,g_N)$ we call \eqref{eq:oas} the overall system of the subsystems~\eqref{eq:subsystem}.
In this sense~\eqref{eq:subsystem} is what we call a large-scale system.

A typical assumption of classical (Lyapunov-based) small-gain theorems is the requirement that each subsystem~\eqref{eq:subsystem} has to admit an ISS Lyapunov function (see e.g. \cite{JLW04-SGdiscrete,DRW09,LHJ12}). 
This assumption comes from the fact that in small-gain theory the influence of the other subsystems is considered as disturbance. 
This is quite conservative as can be seen from the following linear system
\begin{equation*} x(k+1)= \left( \begin{array}{cc} 1.5 & 1 \\ -2 & -1 \end{array} \right) x(k), \qquad  k \in \N.
\end{equation*}
As the spectral radius of the matrix is equal to $\sqrt{2}/2$
the origin is GAS.
But the first decoupled subsystem $x_1(k+1) = 1.5 x_1(k)$ is unstable. 
So classical small-gain theorems cannot be applied.

The aim of this section is to derive an ISS small-gain theorem which relaxes the assumption of classical (Lyapunov-based) ISS small-gain theorems.
We assume that each system has to admit a Lyapunov-type function that
decreases after a finite number of time steps rather than at each time
step. 
Here it is important to note that these decrease conditions are formulated
for the interconnected system, not for the decoupled system. The latter
is a standard feature of the classical approach.
So the Lyapunov-type functions proposed here do not require 0-GAS of the origin of the subsystems. 

The results in this section are based on the small-gain theorems presented in \cite{GL12-CDC,GLW13-SG} for systems without inputs, and the construction of (finite-step) ISS Lyapunov functions presented in \cite{DRW09}.
The section is divided in two parts. 
In  Section~\ref{subsec:dissSGT}, we prove ISS of system~\eqref{eq:oas} by constructing an overall dissipative finite-step ISS Lyapunov function. 
In Section~\ref{subsec:converseSGT}, we show that the relaxed small-gain theorems derived are necessary at least for expISS systems.

%%%%%%%%%%%%%%%%%%%%%%%%%%%%%
\subsection{Dissipative ISS small-gain theorems}
\label{subsec:dissSGT}
%%%%%%%%%%%%%%%%%%%%%%%%%%%%%

We start with the case that the effect of the external input $u$ can be captured via maximization.

\begin{theorem} \label{theo:SGT-Dissipative_v1}
Let \eqref{eq:oas} be 
given by the interconnection of the subsystems in~\eqref{eq:subsystem}.
Assume that there exist an $M \in \N$, $M \geq 1$, functions $\mathcal V_i:\R^{n_i} \rightarrow \R_+$,  $\gamma_{ij} \in \Ki \cup \{0\}$, $\gamma_{iu} \in \K\cup \{0\}$, and positive definite functions $\delta_i $, with  $d_i:=(\id+\delta_i)\in \Ki$, for $i,j \in \{1, \ldots, N\}$ such that with $\Gamma_{\oplus}$ defined in \eqref{eq:gainmap}, and the diagonal operator $D$ defined by $D= \diag (d_i)$ the following conditions hold.
\begin{enumerate}
\item For all $i\in\{1,\ldots, N\}$, the functions $\mathcal V_i$ are proper and positive definite.
\item For all $\xi = (\xi_1, \ldots, \xi_N)\in \R^n$ with $\xi_i \in \R^{n_i}$, $i \in \{1, \ldots, N\}$ and $u(\cdot) \in l^\infty (\R^m)$ it holds that
\begin{equation*} \left[ \begin{array}{c} \mathcal V_1(x_1(M,\xi,u(\cdot))) \\ \vdots \\ \mathcal V_N(x_N(M, \xi,u(\cdot))) \end{array} \right] \leq 
\max \left\{ \Gamma_\oplus \left( \left[ \begin{array}{c} \mathcal V_1(\xi_1) \\ \vdots \\ \mathcal V_N(\xi_N) \end{array} \right] \right) ,  \left[ \begin{array}{c} \mathcal \gamma_{1u}(\dabsu{u}) \\ \vdots \\ \mathcal \gamma_{Nu}(\dabsu{u}) \end{array} \right] \right\}.
\end{equation*}
\item The \emph{small-gain condition\footnote{Note that the strong small-gain condition in Definition~\ref{Def:SGC} requires the functions $\delta_i$ in the diagonal operator to be of class $\Ki$, whereas here we only require $\delta_i$ to be positive definite.}} 
$\Gamma_\oplus \circ D \not \geq \id$
holds.
\end{enumerate}
Then there exists an $\Omega$-path  $\tilde \sigma \in \Ki^N$ for $\Gamma_\oplus \circ D$. 
Moreover,
if for all $i \in \{1, \ldots, N\}$ there exists a $\Ki$-function $\hat \alpha_i$ satisfying 
\begin{equation}\label{eq:alphahat}
\tilde \sigma_i^{-1}\circ d_i^{-1}\circ \tilde \sigma_i =\tilde \sigma_i^{-1}\circ (\id+\delta_i)^{-1}\circ \tilde \sigma_i = \id - \hat \alpha_i
\end{equation} 
then the function $V: \R^n \to \R_+$ defined by
\begin{equation} \label{eq:LF-dissipative} 
V(\xi):= \max_i (\tilde \sigma_i^{-1} \circ d_i^{-1})(\mathcal V_i(\xi_i)).
\end{equation}
is a dissipative finite-step ISS Lyapunov function for system~\eqref{eq:oas}.
In particular, system~\eqref{eq:oas} is ISS.
\end{theorem}

Let us briefly discuss the differences in the assumptions of Theorem~\ref{theo:SGT-Dissipative_v1} when compared to the small-gain theorem for
the maximization case e.g. from \cite{DRW09}. First there is of course the
finite-step condition (ii), but this is not surprising. What is
conceptually new is that a strong small-gain condition $\Gamma_\oplus
\circ D \not \geq \id$ is required and in addition we need the existence of the
functions $\hat \alpha_i$ satisfying \eqref{eq:alphahat}. This is by no
means automatic. If ISS of all subsystems is assumed as in~\cite{DRW09}
then in the maximization case the condition $\Gamma_\oplus
\not \geq \id$ is sufficient. 
In our proof we use the stronger condition and we suspect
  that this extra robustness required is not merely a result of our
  technique of proof. Rather, it is essential to deal with possibly unstable
  subsystems.  We explicitly point out the step where the new condition
\eqref{eq:alphahat} is needed in the proof.

\begin{proof}
Assume that $\mathcal V_i$ and $\gamma_{ij},\gamma_{iu}$ satisfy the hypothesis of the theorem. 
Denote $\gamma_u(\cdot) := (\gamma_{1u}(\cdot), \ldots, \gamma_{Nu}(\cdot))^\top$. 
Then from condition~(iii) and \cite[Theorem~5.2-(iii)]{DRW09} it follows that there exists an $\Omega$-path $\tilde \sigma \in \Ki^N$ 
such that 
\begin{equation*}
 (\Gamma_\oplus \circ D)(\tilde \sigma(s))<\tilde \sigma(s)
\end{equation*} 
holds for all $s>0$.
In particular,
\begin{equation}\label{eq:pathmax-equi}
\max_{i,j \in \{1, \ldots, N\}} \tilde \sigma_i^{-1} \circ \gamma_{ij} \circ d_j \circ \tilde \sigma_j<\id.
\end{equation}
In the following let $i,j \in \{1, \ldots, N\}$.
Let $V:\R^n\rightarrow \R_+$ be defined as in~\eqref{eq:LF-dissipative}.
The aim is to show that $V$ is a dissipative finite-step ISS Lyapunov function for the overall system~\eqref{eq:oas}. 
Recall that condition~(i)
implies the existence of $\alpha_{1i},\alpha_{2i}\in \Ki$ such that for all $\xi_i \in \R^{n_i}$ we have $\alpha_{1i}(\dabs{\xi_i}) \leq \mathcal V_i(\xi_i) \leq \alpha_{2i}(\dabs{\xi_i})$.
Thus,
\begin{equation*} 
V(\xi) \geq \max_i (\tilde \sigma_i^{-1}\circ d_i^{-1}) (\alpha_{1i}(\dabs{\xi_i})) \geq \alpha_1(\dabs{\xi})
\end{equation*}
with $\alpha_1:= \min_{j} \tilde \sigma_j^{-1}\circ d_j^{-1} \circ \alpha_{1j}\circ  \tfrac1\kappa  \id \in \Ki$, where $\kappa\geq 1$ comes from~\eqref{eq:norm}. 
On the other hand we have
\begin{equation*} 
V(\xi) \leq \max_i (\tilde \sigma_i^{-1}\circ d_i^{-1}) (\alpha_{2i}(\dabs{\xi_i})) \leq \alpha_2(\dabs{\xi})
\end{equation*}
with $\alpha_2 := \max_i \tilde \sigma_i^{-1}\circ d_i^{-1} \circ \alpha_{2i} \in \Ki$, which shows $V$ defined in~\eqref{eq:LF-dissipative} is proper and positive definite.
To show the decay of $V$, i.e., an inequality of the form~\eqref{eq:decayV},  we define $\sigma:= \max_i \tilde \sigma_i^{-1} \circ d_i^{-1}\circ \gamma_{iu}$, and obtain
\begin{align*}
V(x(M,\xi, u(\cdot))) &= \max_i (\tilde \sigma_i^{-1} \circ d_i^{-1})(\mathcal V_i(x_i(M,\xi, u(\cdot)))) \\
&\hspace{-2.5cm}\stackrel{\text{cond. }(ii)}{\leq} \max_{i} (\tilde \sigma_i^{-1} \circ d_i^{-1})   \left( \max \left \{ \max_j \gamma_{ij}(\mathcal V_j(\xi_j)), \gamma_{iu}(\dabsu{u}) \right\} \right)\\
&\hspace{-2cm}=\max \left \{ \max_{i,j} \tilde \sigma_i^{-1} \circ d_i^{-1} \circ \gamma_{ij} (\mathcal V_j(\xi_j)), \max_i \tilde \sigma_i^{-1} \circ d_i^{-1} \circ \gamma_{iu}(\dabsu{u}) \right\}\\
&\hspace{-2cm}\leq \max \bigg\{ \max_{i,j} \underbrace{\left(\tilde \sigma_i^{-1} \!\circ\! d_i^{-1} \!\circ\! \tilde \sigma_i\right)}_{=\id-\hat \alpha_i \ \text{by \eqref{eq:alphahat}}}\circ \underbrace{\left(\tilde \sigma_{i}^{-1} \!\circ\! \gamma_{ij} \!\circ\! d_j \!\circ\! \tilde \sigma_j\right)}_{<\id \ \text{by \eqref{eq:pathmax-equi}}} \circ \underbrace{\left(\tilde \sigma_{j}^{-1} \!\circ\! d_{j}^{-1} \!\circ\! \mathcal V_{j}(\xi_{j})\right)}_{\leq V(\xi)}, \sigma(\dabsu{u}) \bigg\}
\\
&\hspace{-2cm}< \max \left \{ \max_{i} (\id-\hat \alpha_i)(V(\xi)), \sigma(\dabsu{u}) \right\}.
\end{align*}
Define $\rho:= \max_i (\id-\hat \alpha_i)$, then $\rho\in \Ki$ by~\eqref{eq:alphahat}, and satisfies $\id-\rho=\min_i \hat \alpha_i \in \Ki$.
Noticing that the maximum can be upper bounded by summation, this shows that $V$ is a dissipative finite-step ISS Lyapunov function as defined in Definition~\ref{def:ftISSLF}. 
Then from Theorem~\ref{theo:ftLF-ISS} we conclude that system~\eqref{eq:oas} is ISS.
\hfill~\end{proof}

\begin{remark}\label{rem:nonISSsubsystems}
(i)~To understand the assumptions imposed in Theorem~\ref{theo:SGT-Dissipative_v1} consider the case that $M=1$ and $\delta_i \in \Ki$, $i\in \{1, \ldots, N\}$.
First, by condition~(ii), we have for any $i \in \{1, \ldots,N\}$
\begin{equation}\label{eq:SGT-M=1}
\mathcal V_i(x_i(1,\xi,u(\cdot))) \leq \max \left\{ \max_{j \in \{1, \ldots, N\}} \gamma_{ij} (\mathcal V_j(\xi_j)), \gamma_{iu}{\dabsu{u}} \right\}.
\end{equation}
From the small-gain condition~(iii) we conclude that $\gamma_{ii} \circ (\id+\delta_i)<\id$ by considering the $i$th unit vector. 
Hence, since $\delta_i \in \Ki$, we have  $\gamma_{ii}<(\id+\delta_i)^{-1}= \id-\hat \delta_i$  with $\hat \delta_i \in \Ki$, where the last equality follows from \cite[Lemma~2.4]{Ruf09Monotone}. Thus, we can write~\eqref{eq:SGT-M=1} as
\begin{equation*}
\mathcal V_i(x_i(1,\xi,u(\cdot))) \leq (\underbrace{\id-\hat \delta_i}_{=:\rho_i})(\mathcal V_i(\xi_i)) + 
\sum_{\substack{j=1\\j\neq i}}^N \gamma_{ij} (\mathcal V_j(\xi_j))
+ \gamma_{iu}{\dabsu{u}}.
\end{equation*}
Together with condition~(i) this implies that the functions $\mathcal V_i$ are dissipative ISS Lyapunov functions for the subsystem~\eqref{eq:subsystem} with respect to both internal and external inputs.
Therefore, if $M=1$ and $\delta_i \in \Ki$ then Theorem~\ref{theo:SGT-Dissipative_v1} is a dissipative small-gain theorem for discrete-time systems in the classical sense, see also Remark~~\ref{rem:references}.

(ii)~If $M=1$ and the functions $\delta_i$ are only positive definite then the functions $\mathcal V_i$ are not necessarily dissipative ISS Lyapunov functions, as we cannot ensure that the decay of $\mathcal V_i$ in terms of the function $\rho_i$ satisfies $\id-\rho_i\in \Ki$. 
Thus, in Theorem~\ref{theo:SGT-Dissipative_v1}, even in the case $M=1$, we do not necessarily assume that the subsystems are ISS.

(iii)~Now consider the case $M>1$.
In Theorem~\ref{theo:SGT-Dissipative_v1}, the internal inputs $x_j$ may have an stabilizing effect on system $x_i$ in the first $M$ time steps, whereas the external input $u$ is considered as a disturbance.
Thus, the subsystems do not have to be ISS, while the overall system is ISS.
This observation is essential as it extends the classical idea of small-gain theory.
In particular, the subsystems~\eqref{eq:subsystem} can be $0$-input unstable, i.e., the origin of the system $x_i(k+1)=g_i(0,\ldots, 0,x_i(k),0, \ldots, 0)$ can be unstable. See also Section~\ref{sec:example} devoted to the discussion of an example.
\end{remark}

In the following example we show that condition~\eqref{eq:alphahat}, which quantifies the robustness given by the scaling matrix $D$, is not trivially satisfied, even if $\delta_i \in \Ki$.
 
\begin{example} 
Consider the functions
\begin{equation*}
\tilde \sigma(s):=e^s-1, \quad \tilde \sigma^{-1}(s)= \log (s+1), \quad 
\hat \delta(s)=(s+1)\left(1-\left(\frac{1}{s+1}\right)^{\tfrac{1}{s+1}}\right).
\end{equation*}
It is not hard to see that $\tilde \sigma, \tilde \sigma^{-1} \in\Ki$. 
Moreover, by~\cite[Appendix~A.3]{Geiselhart-Diss}, also $\hat \delta \in \Ki$ and $ (\id-\hat \delta) \in \Ki$.
Similarly as in \cite[Lemma~2.4]{Ruf09Monotone}, there exists\footnote{Note that \cite[Lemma~2.4]{Ruf09Monotone} argues that if $\delta\in \Ki$ is given, there exists a suitable $\hat \delta \in \Ki$ satisfying $(\id+ \delta)^{-1} = \id - \hat \delta$. 
Conversely, for any $\hat \delta \in \Ki$ with $(\id-\hat \delta) \in \Ki$ the existence of a $\delta \in \Ki$ satisfying $(\id-\delta)^{-1} = \id- \hat \delta$
follows by defining $\delta=\hat \delta\circ(\id-\hat \delta)^{-1}\in \Ki$.} a function $\delta \in \Ki$ such that $(\id+\delta)^{-1}=\id-\hat \delta$. 
Hence, we have for all $s \in \R_+$
\begin{equation*}
\tilde \sigma^{-1} \circ(\id+\delta)^{-1}\circ \tilde \sigma(s) = \tilde \sigma^{-1} \circ(\id-\hat\delta)\circ \tilde \sigma(s) = s(1-e^{-s}).
\end{equation*}
As $\lim_{s \to \infty} s (1-e^{-s})-s = 0$, there cannot exist a $\Ki$-function $\hat \alpha$ satisfying~\eqref{eq:alphahat}.
\end{example}

In condition~(ii) of Theorem~\ref{theo:SGT-Dissipative_v1} the effect of internal and external inputs is captured via maximization.
Next, we replace the maximum in condition~(ii) of Theorem~\ref{theo:SGT-Dissipative_v1} by a sum.
Note that in the case of summation, the small-gain condition invoked in Theorem~\ref{theo:SGT-Dissipative_v1} is not strong enough to ensure that $V$ defined in~\eqref{eq:LF-dissipative} is a dissipative finite-step ISS Lyapunov function (see \cite{DRW09}), so we also have to change condition~(iii) of Theorem~\ref{theo:SGT-Dissipative_v1}. 
In particular, we  assume that the functions $\delta_i $ are of class $\Ki$, and not only positive definite.
We recall from Section~\ref{subsec:SGC} that if the diagonal operator $D=\diag(\id+ \delta_i)$ is factorized into 
\begin{equation}\label{eq:Dsplit}
D=D_{2}\circ D_{1}, \qquad D_j=\diag(\id+\delta_{ij}),\, \delta_{ij} \in \Ki, \, i\in \{1, \ldots,N\}, \, j \in \{1,2\}
\end{equation}
then $D\circ \Gamma_\oplus \not \geq \id$ is equivalent to $D_1 \circ \Gamma_\oplus \circ D_2 \not \geq \id$.

\begin{theorem} \label{theo:SGT-Dissipative_v2}
Let~\eqref{eq:oas} be 
given by the interconnection of the subsystems in~\eqref{eq:subsystem}.
Let $\delta_i,\delta_{i_1},\delta_{i_2}\in\Ki$ for $i\in\{1, \ldots,N\}$ and $D:= \diag (d_i) := \diag (\id +\delta_i)$ satisfy~\eqref{eq:Dsplit}.
Assume that there exist an $M \in \N$, $M \geq 1$, functions $\mathcal V_i:\R^{n_i} \rightarrow \R_+$,  
$\gamma_{ij} \in \Ki \cup \{0\}$, and $\gamma_{iu} \in \K\cup \{0\}$ for $i,j \in \{1, \ldots, N\}$ such that with $\Gamma_{\oplus}$ defined in~\eqref{eq:gainmap} the following conditions hold.
\begin{enumerate}
\item For all $i\in\{1,\ldots, N\}$, the functions $\mathcal V_i$ are proper and positive definite.
\item For all $\xi \in \R^n$ and $u(\cdot) \in l^\infty (\R^m)$ it holds that
\begin{equation*} \left[ \begin{array}{c} \mathcal V_1(x_1(M,\xi,u(\cdot))) \\ \vdots \\ \mathcal V_N(x_N(M, \xi,u(\cdot))) \end{array} \right] \leq 
 \Gamma_\oplus \left( \left[ \begin{array}{c} \mathcal V_1(\xi_1) \\ \vdots \\ \mathcal V_N(\xi_N) \end{array} \right] \right) +  \left[ \begin{array}{c} \mathcal \gamma_{1u}(\dabsu{u}) \\ \vdots \\ \mathcal \gamma_{Nu}(\dabsu{u}) \end{array} \right] .
\end{equation*}
\item The \emph{strong small-gain condition}  
$D\circ \Gamma_\oplus \not \geq \id$ 
is satisfied.
\end{enumerate}
Then  there exists an $\Omega$-path $\tilde \sigma \in \Ki^N$ for $D_1\circ \Gamma_\oplus \circ D_2$.
Moreover, if for all $i \in \{1, \ldots, N\}$ there exists a $\Ki$-function $\hat \alpha_i$ satisfying
\begin{equation}\label{eq:alphahat2}
\tilde \sigma_i^{-1}\circ d_{i2}^{-1}\circ \tilde \sigma_i= \id - \hat \alpha_i
\end{equation} 
 then the function $V: \R^n \rightarrow \R_+$ defined by 
\begin{equation}\label{eq:LF-dissipative2}
V(\xi):= \max_{i} (\tilde \sigma_i^{-1}\circ d_{i2}^{-1})(\mathcal V_i(\xi_i))
\end{equation}
is a dissipative finite-step ISS Lyapunov function for system~\eqref{eq:oas}.
In particular, system~\eqref{eq:oas} is ISS.
\end{theorem}

Again note the difference in the assumptions when compared to the
small-gain theorem for the summation case e.g. from \cite{DRW09}. Now a
factorization property for the operator $D$ is required. This imposes
extra conditions as discussed in Remark~\ref{Geiselhart-Diss}. Again these
extra conditions appear to be necessary
to treat the potential instability of subsystems.

\begin{proof} 
In Section~\ref{subsec:SGC} we have shown  that $D \circ \Gamma_\oplus \not \geq \id$ if and only if  $D_1 \circ \Gamma_\oplus \circ D_2 \not \geq \id$.
By~\cite[Theorem~5.2-(iii)]{DRW09} it follows that there exists an $\Omega$-path  $\tilde \sigma \in \Ki^N$  for $D_1 \circ \Gamma_\oplus \circ D_2$ satisfying
\begin{equation*}
(D_1 \circ \Gamma_\oplus \circ D_2)(\tilde \sigma(r))<\tilde \sigma(r) \qquad \forall r>0,
\end{equation*}
or, equivalently, for all  $i \in \{1, \ldots, N\} $,
\begin{equation}\label{eq:pathinequality1}
\max_{j \in \{1, \ldots, N\}} d_{i1} \circ \gamma_{ij} \circ d_{j2} \circ \tilde \sigma_j(r)< \tilde \sigma_i(r) \qquad \forall r>0.\end{equation}
We show that this inequality implies the existence of a function $\varphi \in \Ki$ such that
 for all $r>0$, we have
\begin{equation} \label{eq:pathestimate}
\max_{i,j} \tilde  \sigma_i^{-1} \circ \left( \gamma_{ij}\circ d_{j2} \circ \tilde \sigma_j (r) + \gamma_{iu}\circ \varphi (r) \right) <r .
\end{equation}
To do this, we assume, without loss of generality\footnote{
If $\gamma_{iu} \in \K \backslash \Ki$ take any $\Ki$-function upper bounding $\gamma_{iu}$. If $\gamma_{iu}=0$, take e.g. $\gamma_{iu}=\id$.
}, that $\gamma_{iu} \in \Ki$.
Since $d_{i1} = \id + \delta_{i1}$ with $\delta_{i1}\in \Ki$ for all $i \in \{1, \ldots, N\}$, we can write~\eqref{eq:pathinequality1} as
\begin{equation}\label{eq:pathinequality2}
\max_{j \in \{1, \ldots, N\}} \gamma_{ij} \circ d_{j2} \circ \tilde \sigma_j(r) +
\max_{j \in \{1, \ldots, N\}} \delta_{i1} \circ \gamma_{ij} \circ d_{j2} \circ \tilde \sigma_j(r)< \tilde \sigma_i(r).
\end{equation}
Let $i \in \{1, \ldots, N\}$. We consider two cases:
\begin{enumerate}
\item If $\gamma_{ij}=0$ for all $j \in \{1, \ldots, N\}$, define
\begin{equation*}
\varphi_i := \tfrac12 \gamma_{iu}^{-1} \circ \tilde \sigma_i \in \Ki.
\end{equation*}
\item If $\gamma_{ij} \in \Ki$ for at least one $j\in \{1, \ldots, N\}$, define
\begin{equation*}
\varphi_i:= \max_{j \in \{1, \ldots, N\}} \gamma_{iu}^{-1} \circ \delta_{i1} \circ \gamma_{ij} \circ d_{j2} \circ \tilde \sigma_j \in \Ki.
\end{equation*}
Note that we need $\delta_{i1} \in \Ki$ for $\varphi_i$ to be of class $\Ki$, as opposed to proof of Theorem~\ref{theo:SGT-Dissipative_v1}, where we only needed positive definiteness.
\end{enumerate}
For both cases, the definition of $\varphi_i \in \Ki$ together with~\eqref{eq:pathinequality2} implies
\begin{equation*}
\max_{j \in \{1, \ldots, N\}} \gamma_{ij} \circ d_{j2} \circ \tilde \sigma_j(r) +
\gamma_{iu} \circ \varphi_i(r)< \tilde \sigma_i(r).
\end{equation*}
for all $r >0$.
Then it is not hard to see that $\varphi:= \min_{i \in \{1, \ldots, N\}} \varphi_i \in \Ki$ satisfies~\eqref{eq:pathestimate} for all $r>0$.
 
In the following let $i,j \in \{1, \ldots, N\}$.
Consider the function $V$ from~\eqref{eq:LF-dissipative2}. 
First note that $V$ is proper and positive definite, which follows directly from the proof of Theorem~\ref{theo:SGT-Dissipative_v1}.
To show the decay of $V$, i.e., an inequality of the form~\eqref{eq:decayV}, we use~\eqref{eq:pathestimate}, and obtain
\begin{align*}
V(x(M,\xi, u(\cdot))) &= \max_i (\tilde \sigma_i^{-1} \circ d_{i2}^{-1})(\mathcal V_i(x_i(M,\xi, u(\cdot)))) \\
&\hspace{-2.3cm}\stackrel{\text{cond. }(ii)}{\leq} \max_{i} (\tilde \sigma_i^{-1} \circ d_{i2}^{-1})   \left(  \max_j \gamma_{ij}(\mathcal V_j(\xi_j))+ \gamma_{iu}(\dabsu{u})\right)\\
&\hspace{-2.1cm}\stackrel{\eqref{eq:alphahat2}}{=}  \max_{i,j} (\id-\hat\alpha_i)\circ \tilde \sigma_i^{-1} \left( \gamma_{ij} (\mathcal V_j(\xi_j))+ \gamma_{iu}(\dabsu{u}) \right)\\
&\hspace{-1.9cm}= \max_{i,j} (\id-\hat\alpha_i)\!\circ\! \tilde \sigma_i^{-1}\!\! \left(\! \gamma_{ij}\!\circ\! d_{j2}\!\circ\! \tilde \sigma_j \!\circ\! \underbrace{\left(\tilde \sigma_{j}^{-1} \!\circ\! d_{j2}^{-1} \circ \mathcal V_{j}(\xi_{j'})\right)}_{ \leq V(\xi)} + \gamma_{iu}\!\circ\! \varphi \!\circ\! \varphi^{-1}(\dabsu{u}) \right)\\
&\hspace{-2.1cm}\stackrel{\eqref{eq:pathestimate}}{<} \max_{i} (\id-\hat\alpha_i) \left(\max\{V(\xi), \varphi^{-1}(\dabsu{u})\} \right)\\
&\hspace{-1.9cm}\leq \max_{i} (\id-\hat\alpha_i) (V(\xi)) + \max_{i} (\id-\hat\alpha_i) ( \varphi^{-1}(\dabsu{u})).
\end{align*}
Define $\rho:= \max_i (\id-\hat \alpha_i)$ and $\sigma:=\max_{i} (\id-\hat \alpha_i) \circ \varphi^{-1}$, then $\id-\rho=\min_{i} \hat \alpha_i \in \Ki$. Hence,~\eqref{eq:decayV} is satisfied.
Again, as in the proof of Theorem~\ref{theo:SGT-Dissipative_v1}, this shows that $V$ is a dissipative finite-step ISS Lyapunov function as defined in Definition~\ref{def:ftISSLF}. 
From Theorem~\ref{theo:ftLF-ISS} we conclude that system~\eqref{eq:oas} is ISS.
\hfill~\end{proof}

\begin{remark}\label{rem:references}
If Theorem~\ref{theo:SGT-Dissipative_v1} (resp. Theorem~\ref{theo:SGT-Dissipative_v2}) is satisfied with $M=1$, then the dissipative finite-step ISS Lyapunov function $V$ in \eqref{eq:LF-dissipative} (resp. \eqref{eq:LF-dissipative2}) is a dissipative ISS Lyapunov function.
In particular, we obtain the following special cases: 
If Theorem~\ref{theo:SGT-Dissipative_v2} is satisfied for $M=1$ then this gives a dissipative-form discrete-time version of \cite[Corollary~5.6]{DRW09}.
On the other hand, for $M=1$, Theorem~\ref{theo:SGT-Dissipative_v1} includes the ISS variant of \cite[Theorem~3]{JLW08-CCC} as a special case.
\end{remark}

\begin{remark}
\label{Geiselhart-Diss}
Whether or not condition~\eqref{eq:alphahat2} in
Theorem~\ref{theo:SGT-Dissipative_v1} is satisfied 
depends on the factorization~\eqref{eq:Dsplit}.
Let $D_1,D_2$ as well as $\hat D_1, \hat D_2$ be two compositions of $D$ as in~\eqref{eq:Dsplit}, i.e., $D_2 \circ D_1 = D = \hat D_2 \circ \hat D_1$.
A direct computation shows that if $\tilde \sigma \in \Ki^N$ is an $\Omega$-path for $D_1 \circ \Gamma_{\oplus} \circ D_2$ then $\hat \sigma:= \hat D_1 \circ D_1^{-1} \circ \tilde \sigma \in \Ki^N$ is an $\Omega$-path for $\hat D_1 \circ \Gamma_{\oplus} \circ \hat D_2$.
Moreover, if we assume that~\eqref{eq:alphahat2} holds
for $D_1,D_2$,
then we have
\begin{align*}
\hat \sigma_i^{-1} \circ \hat d_{i2}^{-1} \circ \hat  \sigma_i
& = (\tilde \sigma_i^{-1} \circ d_{i2}^{-1} \circ \tilde  \sigma_i ) \circ (\tilde \sigma_i^{-1} \circ \hat d_{i1} \circ d_{i1}^{-1} \circ \tilde \sigma_i) \\
& = (\id - \hat \alpha_i) \circ (\tilde \sigma_i^{-1} \circ \hat d_{i1} \circ d_{i1}^{-1} \circ \tilde \sigma_i)
\end{align*}
with $\hat \alpha_i \in \Ki$, $i \in \{1, \ldots, N\}$. 
Unfortunately, from this equation  we cannot conclude that a condition of the form~\eqref{eq:alphahat2} holds for the decomposition $\hat D_1, \hat D_2$ and the $\Omega$-path $\hat \sigma$ as defined above.

An approach for finding a suitable decomposition is the following. Consider the special case of the decomposition~\eqref{eq:Dsplit} with $D_2 = \diag(\id)$ and $D_1 = D= \diag(d_i)$.
Let $\tilde \sigma \in \Ki^N$ be an $\Omega$-path for $D_1 \circ \Gamma_\oplus \circ D_2 = D \circ \Gamma_\oplus$.
Following the same steps as above, we see that $\hat \sigma = \hat D_1 \circ D^{-1} \circ \tilde \sigma \in \Ki^N$ is an $\Omega$-path for $\hat D_1 \circ \Gamma_\oplus \circ \hat D_2$.
Moreover, as $D= \hat D_2 \circ \hat D_1$ implies $d_i^{-1} = \hat d_{i1}^{-1} \circ \hat d_{i2}^{-1}$, we see that
\begin{align*}
\hat \sigma_i^{-1} \circ \hat d_{i2}^{-1} \circ \hat  \sigma_i
& = (\tilde \sigma_i^{-1} \circ d_i \circ \hat d_{i1}^{-1}) \circ ( \hat d_{i1} \circ d_i^{-1}) \circ (\hat d_{i1} \circ d_{i}^{-1} \circ \tilde \sigma_i) \\
& = (\tilde \sigma_i^{-1} \hat d_{i1} \circ d_{i}^{-1} \circ \tilde \sigma_i) \\
& = \tilde \sigma_i^{-1} \hat d_{i2}^{-1} \circ \tilde \sigma_i.
\end{align*}
Hence, a condition of the form~\eqref{eq:alphahat2} holds for the decomposition $\hat D_1, \hat D_2$ with $\Omega$-path $\hat \sigma \in \Ki^N$ if and only there exist $\Ki$-functions $\hat \alpha_i$, $i \in \{1, \ldots, N\}$ satisfying
\begin{equation}\label{eq:alpha-rem}
\tilde \sigma_i^{-1} \circ \hat d_{i2}^{-1} \circ \tilde \sigma_i = \id - \hat \alpha_i.
\end{equation}
Hence, to find a ``good'' decomposition $\hat D_1, \hat D_2$ of $D$, i.e., a decomposition for which~\eqref{eq:alphahat2} holds, we can try to find $\Ki$-functions $\hat d_{i2}^{-1} = (\id + \hat \delta_{i2})^{-1}$, $i \in \{1, \ldots, N\}$ that satisfy
\begin{enumerate}
\item equation~\eqref{eq:alpha-rem} for suitable $\hat\alpha_i \in \Ki$;
\item $\hat d_{i2} \circ \hat d_{i1} = d_i = \id + \delta_i$ with suitable $\hat d_{i1} = \id + \hat \delta_{i1}$.
\end{enumerate}
It is an open question, how to characterize the cases in which
there exists a decomposition satisfying these two conditions.
\end{remark}

In Theorem~\ref{theo:SGT-Dissipative_v1} we introduced the diagonal operator $D$ and assumed~\eqref{eq:alphahat}.
In the following corollary we impose further assumptions such that we do not need the diagonal operator $D$. 
Under these stronger assumptions, system~\eqref{eq:oas} is shown to be expISS.

\begin{corollary}\label{cor:SGT-Dissipative_v1}
Let~\eqref{eq:oas} be 
given by the interconnection of the subsystems in~\eqref{eq:subsystem}.
Assume there exist an $M\in \N$, $M\geq 1$,  linear functions $\gamma_{ij}\in \Ki$, and functions $\mathcal V_i:\R^{n_i}\rightarrow \R_+$ satisfying condition~(i) of Theorem~\ref{theo:SGT-Dissipative_v1} with linear functions $\alpha_{1i}, \alpha_{2i}$. 
Let condition~(ii) of Theorem~\ref{theo:SGT-Dissipative_v1} hold, and instead of condition~(iii) of Theorem~\ref{theo:SGT-Dissipative_v1} let the small-gain condition~\eqref{eq:sgc} hold.
Furthermore, assume that the $\K$-function $\omega_1$ in Assumption~\ref{ass:omega} is linear. Then system~\eqref{eq:oas} is expISS.
\end{corollary}

\begin{proof}
We follow the proof of Theorem~\ref{theo:SGT-Dissipative_v1}. 
By the small-gain condition~\eqref{eq:sgc} there exists an $\Omega$-path $\tilde \sigma \in \Ki^N$ satisfying 
 $\Gamma_\oplus(\tilde\sigma)(r)<\tilde \sigma(r)$ for all $r >0$, see~\cite{DRW09}.
Moreover, as the functions $\gamma_{ij} \in \Ki$ are linear for all $i,j \in \{1, \ldots, N\}$, we can also assume the $\Omega$-path functions $\tilde \sigma_i \in \Ki$ to be linear, see~\cite{GW12-MCSS}.
Thus, the  function 
\begin{equation} \label{eq:Vexp}
V(\xi):= \max_i \tilde \sigma^{-1}_i(\mathcal V_i(\xi_i))
\end{equation} 
has linear bounds $\alpha_1$ and $\alpha_2$. Furthermore, since $\tilde \sigma_i$ and $\gamma_{ij}$ are linear functions, 
we obtain~\eqref{eq:decayV}
with the linear function $\rho:= \max_{i,j} \tilde \sigma_i^{-1} \circ \gamma_{ij} \circ \tilde \sigma_j<\id$, and $\sigma:= \max_i \tilde \sigma^{-1}_i \circ \gamma_{iu}$.
Clearly, $(\id-\rho)\in \Ki$ by linearity of $\rho$. Thus, $V$ is a dissipative finite-step ISS Lyapunov function for system~\eqref{eq:oas}.
Since $\omega_1$ in Assumption~\ref{ass:omega} is linear, we can apply Theorem~\ref{theo:ftLF-expISS} to show that system~\eqref{eq:oas} is expISS.
\hfill~\end{proof}

By Remark ~\ref{rem:omega}, the requirement that $\omega_1$
  is linear is necessary for the system to be expISS.

A similar reasoning as in Corollary~\ref{cor:SGT-Dissipative_v1} applies in the case, where the external input enters additively.

\begin{corollary}\label{cor:SGT-Dissipative_v2}
Let~\eqref{eq:oas} be 
given by the interconnection of the subsystems in~\eqref{eq:subsystem}.
Assume there exist an $M\in \N$, $M\geq 1$,  linear functions $\gamma_{ij}\in \Ki$, and functions $\mathcal V_i:\R^{n_i}\rightarrow \R_+$ satisfying condition~(i) of Theorem~\ref{theo:SGT-Dissipative_v2} with linear functions $\alpha_{1i}, \alpha_{2i}$. 
Let condition~(ii) of Theorem~\ref{theo:SGT-Dissipative_v2} and the small-gain condition~\eqref{eq:sgc} hold.
Furthermore, assume that $\omega_1$ in Assumption~\ref{ass:omega} is linear. 
Then system~\eqref{eq:oas} is expISS.
\end{corollary}

\begin{proof} We omit the details as the proof follows the lines of the proof of Theorem~\ref{theo:SGT-Dissipative_v2} combined with the argumentation of the proof of Corollary~\ref{cor:SGT-Dissipative_v1}.

First, the small-gain condition implies the existence of an $\Omega$-path $\tilde \sigma \in \Ki^N$, which is linear as the functions $\gamma_{ij}$ are linear. In particular, $\Gamma_\oplus(\tilde \sigma(r))< \sigma(r)$ for all $ r>0$.
Next, note that the function $V$ defined in~\eqref{eq:Vexp} has linear bounds as shown in the proof of Corollary~\ref{cor:SGT-Dissipative_v1}.
It satisfies 
\begin{equation*}
V(x(M,\xi,u(\cdot))) \leq \rho(V(\xi)) + \sigma(\dabsu{u})
\end{equation*}
with $\rho:= \max_{i,j} \tilde \sigma_i^{-1} \circ \gamma_{ij} \circ \tilde \sigma_j$ and $\sigma:= \max_i \tilde \sigma_i^{-1} \circ \gamma_{iu}$, which can be seen by a straightforward calculation, invoking condition~(ii) of Theorem~\ref{theo:SGT-Dissipative_v2} and the linearity of the $\Omega$-path~$\tilde \sigma$. 
Again as in the proof of Corollary~\ref{cor:SGT-Dissipative_v1}, $(\id-\rho)\in \Ki$ by linearity of $\rho$. Thus, $V$ defined in~\eqref{eq:Vexp} is a dissipative finite-step ISS Lyapunov function for system~\eqref{eq:oas}.
Since $\omega_1$ in Assumption~\ref{ass:omega} is linear, we can apply Theorem~\ref{theo:ftLF-expISS}, and the result follows.
\hfill~\end{proof}

In this section we have presented sufficient criteria to conclude ISS, whereas in the next section we will study the necessity of these relaxed small-gain results.

%%%
\subsection{Non-conservative expISS Small-Gain Theorems}\label{subsec:converseSGT}
%%%

In the remainder of this section we show that the relaxation of classical small-gain theorems given in Theorems~\ref{theo:SGT-Dissipative_v1} and \ref{theo:SGT-Dissipative_v2} is non-conservative
at least for expISS systems.

\begin{theorem} \label{theo:SGTdissipativeConv} 
Let system~\eqref{eq:oas} be 
given by the interconnection of the subsystems in~\eqref{eq:subsystem}.
Then system~\eqref{eq:oas}
 is expISS if and only if 
\begin{enumerate}
\item
Assumption~\ref{ass:omega} holds with linear $\omega_1$, and
\item
there exist an $M\in \N$, $M\geq 1$,  
linear functions $\gamma_{ij}\in \Ki$, 
proper and positive definite functions $\mathcal V_i:\R^{n_i}\rightarrow \R_+$ with linear bounds $\alpha_{1i}, \alpha_{2i}\in\Ki$  such that the following holds:
\begin{enumerate}
\item condition (ii) of Theorem~\ref{theo:SGT-Dissipative_v1} (and thus also condition~(ii) of Theorem~\ref{theo:SGT-Dissipative_v2});
\item the small-gain condition \eqref{eq:sgc}.
\end{enumerate}
\end{enumerate}
\end{theorem}

\begin{proof}
Sufficiency is shown in Corollary~\ref{cor:SGT-Dissipative_v1} and Corollary~\ref{cor:SGT-Dissipative_v2}, so we only have to prove necessity.
Since system~\eqref{eq:oas} is expISS, global $\K$-boundedness holds with linear $\omega_1$, see Remark~\ref{rem:omega}.
Furthermore, the function $V(\xi):= \dabs{\xi}$, $\xi \in \R^n$ is a dissipative finite-step ISS Lyapunov function for system~\eqref{eq:oas} by Theorem~\ref{theo:converseexpISS}. 
Hence, there exist $\tilde M \in \N$, $\sigma \in \K$ 
and $c<1$ 
such that for all $\xi \in \R^n$ and all $ u(\cdot) \in l^\infty (\R^m)$ we have
\begin{equation}\label{eq:xleqdrho+sigma}
\dabs{x(\tilde M,\xi,u(\cdot))} \leq c\dabs{\xi} + \sigma(\dabsu{u}).
\end{equation}
Define 
$ \mathcal V_i(\xi_i) := \dabs{\xi_i}$
 for $i \in \{1, \ldots, N\}$, where the norm for $\xi_i \in \R^{n_i}$ is defined in the preliminaries. Then $\mathcal V_i$ is proper and positive definite with $\alpha_{1i} = \alpha_{2i} =\id$ for all $i\in \{1, \ldots, N\}$.
Take $\kappa\geq 1$ from~\eqref{eq:norm}, and define
\begin{equation*}
l:=\min\{ \ell \in \N : c^{\ell}\kappa < \tfrac{1}{2} \},
\end{equation*}
which exists as $c \in [0,1)$.
Then we have
\begin{align*}
\mathcal V_i(x_i(l \tilde M,\xi,u(\cdot))) &=  \dabs{x_i(l \tilde M,\xi,u(\cdot))} 
\leq  \dabs{x(l \tilde M,\xi,u(\cdot))} \\
& \stackrel{\eqref{eq:xleqdrho+sigma}}{\leq}   c\dabs{x((l-1) \tilde M,\xi,u(\cdot))} + \sigma(\dabsu{u}) \\
& \stackrel{\eqref{eq:xleqdrho+sigma}}{\leq} c^l \dabs{\xi} + \sum_{j'=1}^l c^{j'-1} \sigma(\dabsu{u}) \\
& \stackrel{\eqref{eq:norm}}{\leq} \max_j  c^l\kappa\dabs{\xi_j} +\sum_{j'=1}^l c^{j'-1} \sigma(\dabsu{u}) \\
& \leq \max \left\{ \max_j \gamma_{ij}(\xi_j),  \gamma_{iu}(\dabsu{u}) \right\} 
\\
& \leq  \max_j \gamma_{ij}(\xi_j) + \gamma_{iu}(\dabsu{u}) 
\end{align*}
with $ \gamma_{iu}(\cdot):= 2 \sum_{j=1}^l c^{j-1} \sigma(\cdot)$, and $\gamma_{ij} := 2c^l\kappa \id$. 
The last inequality shows condition~(ii) of Theorem~\ref{theo:SGT-Dissipative_v2}, while the second last inequality shows condition~(ii) of Theorem~\ref{theo:SGT-Dissipative_v1} for $M= l \tilde M$. 
Finally, by definition of $l\in \N$, we have  $\gamma_{ij}< \id$ for all $i,j \in \{1, \ldots, N\}$. 
Hence, Proposition~\ref{prop:cycle} implies the small-gain condition~\eqref{eq:sgc}. This proves the theorem.
\hfill~\end{proof}

Theorem~\ref{theo:SGTdissipativeConv} is proved in a constructive way, i.e., it is shown that under the assumption that system~\eqref{eq:oas} is expISS we can choose the Lyapunov-type functions $\mathcal V_i: \R^{n_i} \rightarrow \R_+$ as norms, i.e., $\mathcal V_i(\cdot) = \dabs{\cdot}$. 
Then there exist an $M \in \N$ and linear gains $\gamma_{ij} \in \Ki$ satisfying condition~(ii) of Theorem~\ref{theo:SGT-Dissipative_v1}, and thus also condition~(ii) of Theorem~\ref{theo:SGT-Dissipative_v2}, as well as the small-gain condition~\eqref{eq:sgc}. This suggests the following procedure.

\begin{proc}\label{proc:expISS}
Consider~\eqref{eq:oas} as the overall system of the subsystems~\eqref{eq:subsystem}.
Check that Assumption~\ref{ass:omega} is satisfied with a linear $\omega_1$ (else the origin of system~\eqref{eq:oas} cannot be expISS, see Remark~\ref{rem:omega}).
Define $\mathcal V_i(\xi_i) := \dabs{\xi_i}$ for $\xi_i \in \R^{n_i}$, and set $k=1$.
\begin{enumerate} 
\item[{[1]}] Compute $\gamma_{iu} \in \Ki$ and linear functions $\gamma_{ij}\in \Ki \cup \{0\}$ satisfying
\begin{equation*}
\mathcal V_i(x_i(k,\xi,u(\cdot))) = \dabs{x_i(k,\xi,u(\cdot))} \leq \max_{j \in \{1, \ldots, N\}} \gamma_{ij} \dabs{\xi_j} + \gamma_{iu}(\dabsu{u}).
\end{equation*}
\item[{[2]}] Check the small-gain condition~\eqref{eq:sgc} with $\Gamma_{\oplus}$ defined in~\eqref{eq:gainmap}. If~\eqref{eq:sgc} is violated set $k=k+1$ and repeat with~[1].
\end{enumerate}
If this procedure is successful, then expISS of the overall system~\eqref{eq:oas} is shown by 
Theorem~\ref{theo:SGTdissipativeConv}.
Moreover, a dissipative finite-step ISS Lyapunov function can be constructed via~\eqref{eq:Vexp}.
\end{proc}

\begin{remark}
Although Procedure~\ref{proc:expISS} is straightforward, even for simple classes of systems, finding a suitable $M \in \N$  may be computationally intractable, as it was shown in \cite{Blondel99}. 
A systematic way to find a suitable number $M\in \N$ for certain classes of systems is discussed in \cite{GGLW13-SCL}.
\end{remark}
%%%

In the next section we consider a nonlinear system and show how Procedure~\ref{proc:expISS} can be applied.

% % % % % % % % % % % % % % % % % % % % % % % % % % % % % % % % % % % % % % % % % % % % % 
\section{Example} \label{sec:example}

In Section~\ref{sec:sgt} the conservatism of classical small-gain theorems was illustrated by a linear example without external inputs, where the origin is GAS, but 
where the decoupled subsystems' origin are not $0$-GAS. 
In this section we consider a nonlinear example with external inputs and show how the relaxed small-gain theorem from Section~\ref{sec:sgt} can be applied.

Consider the nonlinear system
\begin{equation} \label{eq:ex}
\begin{aligned}
x_1(k+1) &= x_1(k)- 0.3 x_2(k) + u(k)\\
x_2(k+1) &= x_1(k) + 0.3 \frac{x_2^2(k)}{1+x_2^2(k)}
\end{aligned}
\end{equation}
with $ x_1(\cdot), x_2(\cdot), u(\cdot) \in l^\infty (\R)$.
We will show that this system is ISS by constructing a suitable dissipative finite-step ISS Lyapunov function following Procedure~\ref{proc:expISS}. 
Note that the origin of the first subsystem decoupled from the second subsystem with zero input is not 0-GAS,\footnote{We could also make the first system $0$-input unstable by letting $x_1(k+1)=(1+\epsilon)x_1(k)-0.3x_2(k)+u(k)$ and $\epsilon>0$ small enough, and obtain the same conclusion,
 see also \cite{GLW13-SG}. But here we let $\epsilon=0$ to simplify computations.}  hence not ISS. 
So we cannot find an ISS Lyapunov function for this subsystem. 

The converse small-gain results in Section~\ref{subsec:converseSGT} suggest to prove ISS  by a search for suitable functions $\mathcal V_i$ and $\gamma_{ij} \in \Ki$
that satisfy the conditions of one of the small-gains theorems of this section (eg. Theorem~\ref{theo:SGT-Dissipative_v1} or Corollary~\ref{cor:SGT-Dissipative_v1}).
Here we follow Procedure~\ref{proc:expISS}.

First, the right-hand side $G$ of \eqref{eq:ex} is globally $\K$-bounded, since
\begin{equation*}
\dabsinf{G(\xi,\mu)}\leq \max\{\dabs{\xi_1}+0.3 \dabs{\xi_2}+ \dabs{\mu}, \dabs{\xi_1}+0.3 \tfrac{\xi_2^2}{1+\xi_2^2} \} \leq 1.3 \dabsinf{\xi}+\dabs{\mu},
\end{equation*}
where we used
that for all $x \in \R$ we have
\begin{equation}\label{eq:x^2/1+x^2} 
\tfrac{x^2}{1+x^2} \leq \tfrac{\abs{x}}{2}. 
\end{equation}
Let  $\mathcal V_i(\xi_i):= \abs{\xi_i}$, $i\in \{1,2\}$.
Then we compute for all $\xi \in \R^2$,
\begin{align*}
\mathcal V_1(x_1(1, \xi, u(\cdot))) &=  \abs{\xi_1 - 0.3 \xi_2 + u(0)} \leq \max \left\{ 2\mathcal V_1(\xi_1), 0.6 \mathcal V_2(\xi_2) \right\}+ \dabsu{u} , \\
\mathcal V_2(x_2(1, \xi, u(\cdot)))  &= \abs{ \xi_1  +  0.3\tfrac{\xi_2^2}{1+\xi_2^2} } \leq \max\left\{ 2\mathcal V_1(\xi_1), 0.6\tfrac{\mathcal V_2^2(\xi_2)}{1+\mathcal V_2^{2}(\xi_2)} \right\} .
\end{align*}
Since $\gamma_{11}(s) = 2s$, the small-gain condition is violated and we cannot conclude stability. 
Intuitively, this was expected from the above observation that the origin of the first subsystem is not ISS.

Computing solutions $x(k,\xi, u(\cdot))$ with initial condition $\xi \in \R^2$ and input $ u(\cdot) \in l^\infty (\R)$ we see that for $k=3$ we have
\begin{equation*}
\displaystyle
x(3,\xi,u(\cdot))=
\left( \!\!\!  \begin{array}{c} 
\scriptstyle
0.4 \xi_1- 0.21 \xi_2 -0.09 \tfrac{\xi_2^2}{1+ \xi_2^2} -0.09 \tfrac{(\xi_1 +0.3 \tfrac{\xi_2^2}{1+\xi_2^2})^2}{1+(\xi_1 +0.3 \tfrac{\xi_2^2}{1+\xi_2^2})^2} + 0.7 u(0) + u(1)+ u(2)
 \\
\scriptstyle
0.7 \xi_1- 0.3 \xi_2 -0.09 \tfrac{\xi_2^2}{1+ \xi_2^2}\!+\!0.3 \tfrac{\left(\xi_1 - 0.3 \xi_2 +0.3 \tfrac{(\xi_1 +0.3 \tfrac{\xi_2^2}{1+\xi_2^2})^2}{1+(\xi_1 +0.3 \tfrac{\xi_2^2}{1+\xi_2^2})^2}\right)^2}{1+\left(\xi_1 - 0.3 \xi_2 +0.3 \tfrac{(\xi_1 +0.3 \tfrac{\xi_2^2}{1+\xi_2^2})^2}{1+(\xi_1 +0.3 \tfrac{\xi_2^2}{1+\xi_2^2})^2}\right)^2} + u(0) + u(1)
 \end{array} \!\!\!\right).\end{equation*}
Using~\eqref{eq:x^2/1+x^2} we compute
\begin{align*}
\mathcal V_1(x_1(3, \xi, u(\cdot))) &\leq  
0.4\abs{\xi_1} + 0.21 \abs{\xi_2} + \tfrac{0.09}{2} \abs{\xi_2} +  \tfrac{0.09}{2}  \left( \abs{\xi_1} + \tfrac{0.3}{2}  \abs{ \xi_2}\right) \\
& \qquad + 0.7 \abs{u(0)} + \abs{u(1)} + \abs{u(2)}
\\
&= \max\{0.89 \mathcal V_1(\xi_1), 0.5235 \mathcal V_2(\xi_2) \} + 2.7 \dabsu{u},
\\
\mathcal V_2(x_2(3, \xi,u(\cdot))) &\leq 
 0.7 \abs{\xi_1} + 0.3 \abs{\xi_2}+\tfrac{0.09}{2} \abs{\xi_2}+  \tfrac{0.3}{2} \left(\abs{\xi_1} + 0.3 \abs{\xi_2} +\tfrac{0.3}{2} (\abs{\xi_1}\!+\!\tfrac{0.3}{2} \abs{\xi_2})  \right) \\
&\qquad + \abs{u(0)} + \abs{u(1)}
\\
&= \max\{1.745 \mathcal V_1(\xi_1), 0.78675\mathcal V_2(\xi_2) \} +2\dabsu{u}.
\end{align*}
From this we derive the linear functions
\begin{align*} 
\gamma_{11} (s) &= 0.89s, & \quad  \gamma_{12} (s) &= 0.5235s , & \qquad & \gamma_{1u} (s) = 2.7s,  \\
\gamma_{21} (s) &= 1.745s,  &\quad  \gamma_{22} (s) &= 0.78675s, & \qquad & \gamma_{2u} (s) = 2s .
\end{align*}
This yields
the map $\Gamma_\oplus:\R^2_+ \rightarrow \R^2_+$ from \eqref{eq:gainmap} as
\begin{equation}\label{eq:Gammaex}
\Gamma_{\oplus}((s_1,s_2)) = \left( \begin{array}{c} \max\{0.89s_1,0.5235s_2\}  \\ \max\{1.745s_1,0.78675s_2\} \end{array} \right).
\end{equation}
Since $\gamma_{11}< \id, \gamma_{22}<\id$ and $\gamma_{12} \circ \gamma_{21} <\id$, we conclude from the cycle condition, Proposition~\ref{prop:cycle}, that the small-gain condition \eqref{eq:sgc} is satisfied. Hence, from Corollary~\ref{cor:SGT-Dissipative_v2} we can now conclude that the origin of system~\eqref{eq:ex} is expISS.

\begin{remark}\label{rem:ex}
The small-gain results in Section~\ref{sec:sgt}, and in particular Corollary~\ref{cor:SGT-Dissipative_v2}, prove the ISS property of the interconnected system~\eqref{eq:oas} by constructing a dissipative finite-step ISS Lyapunov function. 
The following shows that this construction is straightforward to implement. 

Consider system~\eqref{eq:ex} and the map $\Gamma_\oplus$ derived in~\eqref{eq:Gammaex}.
We use the method proposed in \cite{GW12-MCSS} to compute an $\Omega$-path  $ \tilde \sigma(r):=\left( \begin{smallmatrix}0.5 r \\ 0.9r
\end{smallmatrix} \right)$
that satisfies 
\begin{equation*} \Gamma_\oplus (\tilde \sigma(r)) = \left( \begin{array}{c} 0.47115r  \\ 0.8725r \end{array} \right) < \left( \begin{array}{c} 0.5 r \\ 0.9r \end{array} \right) = \tilde \sigma(r) \end{equation*}
for all $r>0$. 
From the proof of Corollary~\ref{cor:SGT-Dissipative_v2} we can now conclude that
 \begin{equation*}
 V(\xi):= \max_i \tilde \sigma_i^{-1}(\mathcal V_i(\xi_i))= \max \{2\abs{\xi_1}, \tfrac{10}{9} \abs{\xi_2} \} \end{equation*}
is a dissipative finite-step ISS Lyapunov function for the overall system \eqref{eq:ex}. 
In particular, following the proof of Corollary~\ref{cor:SGT-Dissipative_v2}, we compute
\begin{equation*}
\rho(s):= \max_{i,j \in \{1,2\}} \tilde \sigma_i^{-1}\circ \gamma_{ij}\circ \tilde \sigma_j(s) = 0.9695 s,
\end{equation*}
and,
\begin{equation*}
\sigma(s):=\max_{i\in\{1,2\}} \tilde \sigma_i^{-1} \circ \gamma_{iu} = 5.4 s
\end{equation*}
for which $V$ satisfies $V(x(3,\xi,u(\cdot)) \leq \rho(V(\xi))+\sigma(\dabsu{u})$ for all $\xi \in \R^2$.
\end{remark}

\begin{remark}
Although the construction of the dissipative finite-step ISS Lyapunov function via the small-gain approach (Procedure~\ref{proc:expISS}) 
requires the computation of an $\Omega$-path, we believe that for large-scale interconnections the small-gain approach is still more advisable than 
a direct search for a dissipative finite-step ISS Lyapunov function (e.g. following Procedure~\ref{proc:LFdirect}). 
The reason for this belief is that the choice of a suitable natural number $M$ in Procedure~\ref{proc:LFdirect}  might be, in general, much higher than the choice for a suitable natural number $M$ in Procedure~\ref{proc:expISS}. 

For instance, consider system~\eqref{eq:ex}. 
As shown above, Procedure~\ref{proc:expISS} can be applied for $M=3$.
On the other hand, to make computations of norm estimates simpler, consider the $1$-norm $\dabs{\cdot}_1$.
Then we obtain
\begin{equation*}
\dabs{x(3,\xi, u(\cdot))} \leq 1.2225 \abs{\xi_1} + 0.655125 \abs{\xi_2} + 4.7 \dabsu{u}
\end{equation*}
by following similar computations as above using~\eqref{eq:x^2/1+x^2}. 
Hence, $V(\cdot):= \dabs{\cdot}_1$ cannot be a dissipative finite-step ISS Lyapunov function with $M=3$. Similar estimates also show that $V$ is not a dissipative finite-step ISS Lyapunov function for $M<3$. 
Thus,  Procedure~\ref{proc:LFdirect} for $V(\cdot):= \dabs{\cdot}_1$ requires that $M>3$.
\end{remark}

% % % % % % % % % % % % % % % % % % % % % % % % % % % % % % % % % % % % % % % % % % % % % 
\section{Conclusion}
% % % % % % % % % % % % % % % % % % % % % % % % % % % % % % % % % % % % % % % % % % % % % 

In this work we introduced the notion of dissipative finite-step ISS Lyapunov functions as a relaxation of ISS Lyapunov functions. These finite-step ISS Lyapunov functions were shown to be necessary and sufficient to conclude ISS of the underlying discrete-time system. 
In particular, for expISS system, norms are always dissipative finite-step ISS Lyapunov functions.
Furthermore, we stated relaxed ISS small-gain theorems that drop the common assumption of small-gain theorems that the subsystems are ISS. ISS of the overall systems was then proven by constructing a dissipative finite-step ISS Lyapunov function. For the class of expISS systems, these small-gain theorems are shown to be non-conservative, i.e., necessary and sufficient to conclude ISS of the overall system. An example showed how the results can be applied.

 % % % % % % % % % % % % % % % % % % % % % % % % % % % % % % % % % % % % % % % % % % %

\appendix
\section{\color{white}}

The proofs in Section~\ref{sec:ftISS} require the following lemmas.

\subsection{A comparison lemma}\label{app:comparison}

The following lemma is a particular comparison lemma for finite-step dynamics.

\begin{lemma}\label{lem:comparison}
Let $M\in \N\backslash\{0\}$, $L \in \N \cup \{\infty\}$, $k_0 \in \{0, \ldots, M-1\}$, and $y:\N \rightarrow \R_+$ be a function satisfying 
\begin{equation}\label{eq:lemcomparison} 
y\left( (l+1)M+k_0 \right) \leq \chi \left( y(lM+k_0) \right) , \qquad \forall \, l \in \{0, \ldots, L\},
\end{equation}
where $\chi \in \Ki$ satisfies $\chi < \id$. 
Then there exists a $\KL$-function $\beta_{k_0}$ such that the function $y$ also satisfies
\begin{equation*} 
y(lM+k_0) \leq \beta_{k_0}(y(k_0),lM+k_0), \qquad \forall \, l \in \{0, \ldots, L\}.
\end{equation*}

In addition, if~\eqref{eq:lemcomparison} is satisfied for all $k_0 \in \{0, \ldots, M-1\}$
then there exists a $\KL$-function $\beta$ such that with $y_M^{\max}:= \max \{ y(0), \ldots, y(M-1)\}$ we have
\begin{equation*} 
y(k) \leq \beta(y_M^{\max},k), \qquad \forall k \in \{0, \ldots, (L+1)M-1\}.
\end{equation*}
\end{lemma}

\begin{proof}
Let $M \in \N\backslash\{0\}$, $L \in \N \cup \{\infty\}$ and $k_0 \in \{0, \ldots, M-1\}$.
From~\eqref{eq:lemcomparison} and the monotonicity property of $\chi \in \Ki$, we obtain
\begin{equation*}
y\left( (l+1)M+k_0 \right) \leq \chi \left(y(lM+k_0) \right) \leq \ldots \leq \chi^{l+1}(y(k_0))
\end{equation*}
for all $l \in \{0, \ldots, L\}$.
Note that since $\chi<\id$ we have $\chi^l > \chi^{l+1}$, and  $\chi^l(s) \to 0$ as $l \to \infty$ for any $s \in \R_+$.
Define $t_{k_0,l}:=lM+k_0$ and $t_{k_0,l}^+:= (l+1)M+k_0$ for all $l \in \N$. 
Let $\beta_{k_0}: \R_+ \times \R_+ \to \R_+$ be defined by
\begin{equation*}
\beta_{k_0}(s,r) :=
\left\{ \begin{array}{rl} 
 \tfrac{1}{M} \left ( (t_{k_0,0}-r) \chi^{-1}(s) + (r+M-k_0) \id(s) \right ) & r \in [0, t_{k_0,0}), \, s\geq 0 \\ 
 \tfrac{1}{M} \left ( (t_{k_0,l}^+-r) \chi^l(s) + (r-t_{k_0,l}) \chi^{l+1}(s) \right )   & r \in [t_{k_0,l},t_{k_0,l}^+), \, s\geq 0. \end{array} \right. 
\end{equation*}
 Note that this construction is similar to the one proposed in \cite[Lemma~4.3]{JW02}.  
Clearly, $\beta_{k_0}$ is continuous and $\beta_{k_0}(\cdot, r)$ is a $\K$-function for any fixed $r \geq 0$.  
On the other hand, for any fixed $s \geq 0$, $\beta_{k_0}(s, \cdot)$ is an $\L$-function, as it is linear affine on any interval $[t_{k_0,l}, t_{k_0,l}^+]$ and strictly decreasing by
\begin{equation*}
\beta_{k_0}(s, t_{k_0,l}) = \chi^l(s)>\chi^{l+1}(s) = \beta_{k_0}(s, t_{k_0,l}^+).
\end{equation*}
Hence, $\beta_{k_0} \in \KL$.
Moreover, for all $l \in \{0, \ldots, L\}$ we have
\begin{equation*}
y(lM+k_0) \leq \chi^l (y(k_0)) = \beta_{k_0} \left( y(k_0), lM+k_0 \right),
\end{equation*}
which shows the first assertion of the lemma.

Now let~\eqref{eq:lemcomparison} be satisfied for all $k_0 \in \{0, \ldots, M-1\}$.
Define
\begin{equation*}
\beta(s,r) := \max_{k_0\in \{0, \ldots, M-1\}} \beta_{k_0}(s,r),
\end{equation*}
which is again a function of class $\KL$.
For any $k \in \{0, \ldots, (L+1)M-1\}$ there exist unique $l \in \{0, \ldots, L\}$ and $k_0 \in \{0, \ldots, M-1\}$ such that $k=lM+k_0$, and we have
\begin{equation*}
y(k) = y(lM+k_0) \leq \chi^l(y(k_0)) = \beta_{k_0}(y(k_0),lM+k_0) \leq \beta_{k_0}(y_M^{\max}, k) \leq \beta(y_M^{\max},k)
\end{equation*}
with $y_M^{\max}:= \max \{ y(0), \ldots, y(M-1)\}$.
This concludes the proof.
\end{proof}

If the function $\chi$ in Lemma~\ref{lem:comparison} is linear, then the $\KL$-function $\beta$ has a simpler form as we will see in the next lemma.

\begin{lemma}\label{lem:comparisonexp}
Let the assumptions of Lemma~\ref{lem:comparison} be satisfied for all $k_0 \in \{0, \ldots, M-1\}$ with $\chi(s) = \theta s$ and $\theta \in (0,1)$. 
Let $y_M^{\max}:= \max \{y(0), \ldots, y(M-1)\}$, then for all $k \in \{0, \ldots, (L+1)M-1\}$ we have
\begin{equation*}
y(k) \leq  \frac{y_M^{\max}}{\theta} \theta^{k/M}.
\end{equation*}
\end{lemma}

\begin{proof} 
A direct computation yields that for any $k_0 \in \{0, \ldots, M-1\}$, and any $l \in \{0, \ldots, L\}$  we have $ y(lM+k_0) \leq \chi(y((l-1)M+k_0)) \leq \chi^l(y(k_0)) = \theta^l y(k_0)$. 
Hence, for any $k = lM+k_0\leq (L+1)M-1$, with $l \in\{0, \ldots, L\}$ and $k_0 \in \{0, \ldots, M-1\}$ we have
\begin{equation*}
y(k) \leq \max_{k_0 \in \{0, \ldots, M-1\}} \{y(k_0) \theta^l\} \leq y_M^{\max} \theta^{k/M-1}.
\end{equation*}
This proves the lemma.
\end{proof}

% % % % % % % % % % % % % % % % % % % % % % % % % % % % % % 
\subsection{Bounds on trajectories}\label{subsec:appendixbounds}

As noted in Remark~\ref{rem:omega} the requirement on the existence of $\K$-functions $\omega_1,\omega_2$ satisfying \eqref{eq:omega} in Assumption~\ref{ass:omega} is a necessary condition for system~\eqref{eq:oas} to be ISS. 
The following lemma states that
under the assumption of global $\K$-boundedness any trajectory of system~\eqref{eq:oas} has a global $\K$-bound for any time step. 
This result is needed in Theorem~\ref{theo:ftLF-ISS} to show that the existence of a dissipative finite-step ISS Lyapunov function implies ISS of system~\eqref{eq:oas}.

\begin{lemma}\label{lem:bounds} 
Let system~\eqref{eq:oas} satisfy Assumption~\ref{ass:omega}. 
Then for any $j\in \N$  there exist $\K$-functions $\vartheta_j, \zeta_j$ such that for all $\xi \in \R^n$, and all $u(\cdot) \in l^\infty (\R^m)$ we have
\begin{equation}\label{eq:bounds}
\dabs{x(j, \xi, u(\cdot))} \leq \vartheta_j(\dabs{\xi}) + \zeta_j(\dabsu{u}).
\end{equation}
\end{lemma}

Before we prove this lemma we note that, as shown in Remark~\ref{rem:omega}, any trajectory of an ISS system has a \emph{uniform} global $\K$-bound, i.e.,~\eqref{eq:bounds} is satisfied by taking 
$\vartheta_j(\cdot) = \beta(\cdot, 1)$ and  $\zeta_j(\cdot) = \gamma(\cdot)$.
On the other hand, if the system is not globally stable then we cannot find uniform global $\K$-bounds  upper bounding the functions $\vartheta_j\in\K$, $j \in \N$, and $\zeta_j\in\K$, $j \in \N$, in~\eqref{eq:bounds}. 

\begin{proof}
We prove the result by induction. Take any $\xi \in \R^n$ and any input $u(\cdot) \in l^\infty (\R^m)$. 
For $j=0$ we have $\dabs{x(0, \xi, u(\cdot))} = \dabs{\xi}$
satisfying~\eqref{eq:bounds} with $\vartheta_0=\id$ and arbitrary $\zeta_0 \in \K$.
For $j=1$ it follows by Assumption~\ref{ass:omega} that $\dabs{x(1, \xi, u(\cdot))} \leq \omega_1(\dabs{\xi}) + \omega_2(\dabsu{u})$. 
So we can take $\vartheta_1:= \omega_1$ and $\zeta_1:= \omega_2$.\\
Now assume that there exist $\vartheta_j, \zeta_j \in \K$ satisfying~\eqref{eq:bounds}  for some $j \in \N$. 
Then 
\begin{align*}
\dabs{x(j+1, \xi, u(\cdot))} & = \dabs{G(x(j, \xi, u(\cdot)), u(j))} 
\leq \omega_1(\dabs{x(j, \xi, u(\cdot))})+ \omega_2(\dabsu{u}) \\
& \hspace{-0.6cm}\leq \omega_1(\vartheta_j(\dabs{\xi}) + \zeta_j(\dabsu{u})) + \omega_2(\dabsu{u}) 
\\ & \hspace{-0.6cm} 
\leq \omega_1(2\vartheta_j(\dabs{\xi})) + \omega_1(2\zeta_j(\dabsu{u})) + \omega_2(\dabsu{u}) 
=: \vartheta_{j+1}(\dabs{\xi}) + \zeta_{j+1}(\dabsu{u}).
\end{align*}
By induction, the assertion holds for all $j \in \N$.
\hfill~\end{proof}

If the functions $\omega_1,\omega_2$ in~\eqref{eq:omega} are linear then the functions $\vartheta_j, \zeta_j$ in Lemma~\ref{lem:bounds} are also linear, and have an explicit construction in terms of $\omega_1,\omega_2$.

\begin{lemma}\label{lem:boundsexp} 
Let system~\eqref{eq:oas} satisfy Assumption~\ref{ass:omega} with linear functions $\omega_1(s):= w_1 s$ and $\omega_2(s) = w_2 s$, where $w_1,w_2>0$.
 Then~\eqref{eq:bounds} is satisfied with
$\vartheta_j (s) =  w_1^js$ and $\zeta_j= w_2 \sum_{i=0}^{j-1}w_1^{i} s$.
\end{lemma}

\begin{proof} 
The proof follows using the variation of constants formula.
\hfill~\end{proof}

\end{document}